\documentclass[11pt,leqno]{amsart}  

\usepackage{calc}
\usepackage[left=1in,right=1in,top=1in,bottom=1in]{geometry}  
\usepackage{graphicx}
\usepackage{amsmath,amssymb,epsfig}
\usepackage{amsthm}
\usepackage{enumerate}

\newcommand{\R}{{\mathbb R}}
\newcommand{\Z}{{\mathbb Z}}
\newcommand{\Q}{{\mathbb Q}}
\newcommand{\C}{{\mathbb C}}
\newcommand{\T}{{\mathcal T}}

\newcommand{\Ft}{{\mathbb{F}_2}}

\newcommand{\Gal}{{\mathrm{Gal}}}
\newcommand{\Sel}{{\mathrm{Sel}}}
\newcommand{\dimF}{{\mathrm{dim}_{\mathbb{F}_2}}}
\newcommand{\Ftwo}{{\mathbb{F}_2}}
\newcommand{\loc}{{\mathrm{loc}}}
\newcommand{\Zt}{{\mathbb{Z}/2\mathbb{Z}}}

\newcommand{\hatphi}{{ \hat \phi }}
\newcommand{\phihat}{{ \hat \phi }}
\newcommand{\p}{{ \mathfrak{p} }}
\newcommand{\q}{{ \mathfrak{q} }}
\newcommand{\dd}{{ \mathbf{d} }}
\newcommand{\ord}{{ \mathrm{ord} }}
\newcommand{\modtwo}{{ \pmod{2} }}

\newtheorem{thm}{\bf{Theorem}}[section]
\newtheorem{theorem}{\bf{Theorem}}[section]
\newtheorem{cor}[thm]{\bf{Corollary}}
\newtheorem{corollary}[theorem]{\bf{Corollary}}

\newtheorem{proposition}[theorem]{\bf{Proposition}}

\newtheorem{lemma}[theorem]{\bf{Lemma}}
\newtheorem{con}[thm]{\bf{Conjecture}}

\newtheorem*{LemmaSplitMult}{Lemma \ref{splitmult}}

\theoremstyle{definition}
\newtheorem{definition}[theorem]{\bf{Definition}}
\newtheorem{remark}[theorem]{\bf{Remark}}

\newlength{\wid}
\input xy 
\xyoption{all} 

\usepackage[OT2,T1]{fontenc}
\DeclareSymbolFont{cyrletters}{OT2}{wncyr}{m}{n}
\DeclareMathSymbol{\Sha}{\mathalpha}{cyrletters}{"58}

\address{Department of Mathematics, University of Wisconsin-Madison, 480 Lincoln Drive, Madison, Wisconsin 53706}

\email{klagsbru@math.wisc.edu}

\begin{document}

\bibliographystyle{alpha}

\thanks{This paper is based on work conducted by the author as part of his doctoral thesis at UC-Irvine under the direction of Karl Rubin and was supported in part by NSF grants DMS-0457481 and DMS-0757807. The author would like to express his utmost gratitude to Karl Rubin for the guidance and assistance he provided while undertaking this research.}

\renewcommand{\baselinestretch}{1.5} \small\normalsize    

\title[{\tiny Selmer Ranks of Quadratic Twists of Elliptic Curves with Partial Rational Two-Torsion}]{Selmer Ranks of Quadratic Twists of Elliptic Curves with Partial Rational Two-Torsion}
\author{Zev Klagsbrun}

\begin{abstract}
This paper investigates which integers can appear as 2-Selmer ranks within the quadratic twist family of an elliptic curve $E$ defined over a number field $K$ with $E(K)[2] \simeq \Zt$. We show that if $E$ does not have a cyclic 4-isogeny defined over $K(E[2])$, then subject only to constant 2-Selmer parity, each non-negative integer appears infinitely often as the 2-Selmer rank of a quadratic twist of $E$. If $E$ has a cyclic 4-isogeny defined over $K(E[2])$ but not over $K$ , then we prove the same result for 2-Selmer ranks greater than or equal to $r_2$, the number of complex places of $K$. We also obtain results about the minimum number of twists of $E$ with rank $0$, and subject to standard conjectures, the number of twists with rank $1$, provided $E$ does not have a cyclic 4-isogeny defined over $K$.
\end{abstract}

\maketitle

\pagenumbering{arabic}

\section{Introduction}

This paper investigates the integers occurring as 2-Selmer ranks within the quadratic twist family family of a given elliptic curve $E$. Letting $\Sel_2(E/K)$ denote the 2-Selmer group of $E$ (see Section \ref{selmergroups} for the definition), we define the 2-Selmer rank of $E/K$, denoted $d_2(E/K)$, by $$d_2(E/K) =  \dimF  \Sel_2(E/K) - \dimF E(K)[2].$$ 

\begin{definition} For $X \in \R^+$, define a set $$S(X) = \{ \text{Quadratic }F/K : \mathbf{N}_{K/\Q}\mathfrak{f}(F/K) < X \}$$ where $\mathfrak{f}(F/K)$ is the finite part of the conductor of $F/K$. For each $r \in \Z^{\ge 0}$ define a quantity $N_r(E,X)$ by $$N_r(E, X) = \big | \{ F/K \in S(X) : d_2(E^F/K)  = r \} \big |,$$ where $E^F$ is the quadratic twist of $E$ by $F/K$.
\end{definition}

The 2-Selmer rank of $E$ serves as an upper bound for the Mordell-Weil rank of $E$ so understanding its distribution within quadratic twist families helps us understand the rank distribution within the twist family. We are therefore concerned with the limiting behavior of $\displaystyle{\frac{N_r(E, X)}{|S(X)|}}$ as $X \rightarrow \infty$. Recent work by Kane, building on work of Swinnerton-Dyer and Heath-Brown showed that if $E$ is defined over $\Q$, $E(\Q)[2] \simeq \Zt \times \Zt$, and $E$ does not have a cyclic 4-isogeny defined over $\Q$, then $\displaystyle{\frac{N_r(E, X)}{|S(X)|}}$ tends to an explicit non-zero limit $\alpha_r$ for every non-negative integer $r$ such that sum of the $\alpha_r$ over $r \in Z^{\ge 0}$ is equal to 1 \cite{Kane}, \cite{SD}, \cite{HB}. Additional recent work by this author, Mazur, and Rubin proves that if $E(K)[2] = 0$ and $\Gal(K(E[2])/K) \simeq \mathcal{S}_3$, then a similar result holds using a different method of counting after correcting by some local factors arising over totally complex fields \cite{KMR}. 

Far less is known about the behavior $\displaystyle{\frac{N_r(E, X)}{|S(X)|}}$ for curves with $E(K)[2] \simeq \Zt$ and this paper provides a partial answer. Most notably we prove the following: 
\begin{thm}\label{noisogthm}
Let $E$ be an elliptic curve defined over a number field $K$ with $E(K)[2]\simeq \Zt$ that does not possess a cyclic 4-isogeny defined over $K(E[2])$. Then $N_r(E, X) \gg \frac{X}{\log X}$ for all non-negative $r  \equiv d_2(E/K) \pmod{2}$. If $E$ does not have constant 2-Selmer parity, then $N_r(E, X) \gg \frac{X}{\log X}$ for all $r \in \Z^{\ge 0}$.
\end{thm}

This result is similar to earlier results of Mazur and Rubin for curves with $E(K)[2] = 0$.

Constant 2-Selmer parity is a phenomenom exhibited by certain curves $E$, where $d_2(E^F/K) \equiv d_2(E/K) \pmod{2}$ for all quadratic twists $E^F$ of $E$. Dokchitser and Dokchitser have shown that $E/K$ has constant 2-Selmer parity if and only if $K$ is totally imaginary and $E$ acquires everywhere good reduction over an abelian extension of $K$ \cite{DD2}. 


Constant 2-Selmer parity is one of two known obstructions to a non-negative integer $r$ appearing as the 2-Selmer rank of some twist of $E$. A second obstruction can occur when the condition that $E$ not have a cyclic 4-isogeny defined over $K(E[2])$ is relaxed. This author recently exhibited an infinite family of curves over any number field $K$ with a complex place such that $d_2(E^F/K) \ge r_2$ for every curve $E$ in this family and every quadratic $F/K$, where $r_2$ is the number of complex places of $K$ \cite{BddBelow}. This lower-bound phenomenon is not well understood, but appears to be independent of constant 2-Selmer parity.

We are however able to prove the following results in the special case where $E$ has a cyclic 4-isogeny defined over $K(E[2])$ but not over $K$.

\begin{thm}\label{equaltwist} Let $E$ be an elliptic curve defined over a number field $K$ with $E(K)[2]\simeq \Zt$ that does not possess a cyclic 4-isogeny defined over $K$. If $E$ has a twist $E^F$ such that $d_2(E^F/K) = r$, then $N_r(E, X) \gg \frac{X}{\log X}$.
\end{thm}

\begin{thm}\label{balancedmost}
Let $E$ be an elliptic curve defined over a number field $K$ with $E(K)[2]\simeq \Zt$ that does not possess a cyclic 4-isogeny defined over $K$. Then $N_r(E, X) \gg \frac{X}{\log X}$ for all $r \ge r_2$ with $r  \equiv d_2(E/K) \pmod{2}$, where $r_2$ is the number of conjugate pairs of complex embeddings of $K$. If $E$ does not have constant 2-Selmer parity, then $N_r(E, X) \gg \frac{X}{\log X}$ for all $r \ge r_2$.
\end{thm}

Theorem \ref{balancedmost} limits the lower-bound obstruction and shows that the curves presented in \cite{BddBelow} exhibit the worst possible behavior in regard to lower-bound obstruction among curves that do not have a cyclic 4-isogeny defined over $K$. The next theorem further limits the lower-bound obstruction, essentially saying that it can't apply to both $E$ and $E^\prime$ simultaneously, where $E^\prime$ is the curve 2-isogenous to $E$.

\begin{thm}\label{isogpairthm}
Let $E$ be an elliptic curve defined over a number field $K$ with $E(K)[2]\simeq \Zt$ that does not possess a cyclic 4-isogeny defined over $K$.
\begin{enumerate}[(i)]
\item Then either $N_r(E, X) \gg \frac{X}{\log X}$ for all $r \in \Z^{\ge 0}$ with $r  \equiv d_2(E/K) \pmod{2}$ or $N_r(E^\prime, X) \gg \frac{X}{\log X}$ for all $r \in \Z^{\ge 0}$ with $r  \equiv d_2(E/K) \pmod{2}$.
\item If $E$ does not have constant 2-Selmer parity, then we additionally have that either $N_r(E, X) \gg \frac{X}{\log X}$ for all $r \in \Z^{\ge 0}$ with $r  \not \equiv d_2(E/K) \pmod{2}$ or $N_r(E^\prime, X) \gg \frac{X}{\log X}$ for all $r \in \Z^{\ge 0}$ with $r  \not \equiv d_2(E/K) \pmod{2}$.
\item If either $K$ has a real place or $E$ has a place of multiplicative reduction, then $N_r(E, X) \gg \frac{X}{\log X}$ for all $r \in \Z^{\ge 0}$ or $N_r(E^\prime, X) \gg \frac{X}{\log X}$ for all $r \in \Z^{\ge 0}$. (I.E. The choice of $E$ and $E^\prime$ for parts $(i)$ and $(ii)$ can be taken to be the same.) 
\end{enumerate}
\end{thm}

As the 2-Selmer rank of $E$  serves as an upper bound for the rank of $E$, we are able to prove the following corollaries.

\begin{cor}\label{rank0}
Let $E$ be an elliptic curve defined over a number field $K$ with $E(K)[2]\simeq \Zt$ that does not possess a cyclic 4-isogeny defined over $K$. If either $d_2(E/K) \equiv 0  \pmod{2}$ or $E$ does not have constant 2-Selmer parity, then the number of twists $E^F$ of $E$ having Mordell-Weil rank 0 grows at least as fast as $\frac{X}{\log X}$.
\end{cor}

In order to say something about $E$ having twists of rank one, we need to rely on the following well-known conjecture that is a consequence of the Tate-Shafarevich conjecture.
\begin{con}[Conjecture $\Sha T_2(K)$] For every elliptic curve $E$ defined over $K$, $\dimF \Sha(E/K)[2]$ is even.
\end{con}

The  2-Selmer group sits inside the exact sequence $$0\rightarrow E(K)/2E(K) \rightarrow \Sel_2(E/K) \rightarrow \Sha(E/K)[2] \rightarrow 0.$$ As $\dimF E(K)/2E(K) = \text{rank } E/K + \dimF E(K)[2]$, if Conjecture $\Sha T_2(K)$ holds and $d_2(E/K) = 1$, then the Mordell-Weil rank of $E$ will be 1. We can therefore state the following:

\begin{cor}\label{rank1}
Let $E$ be an elliptic curve defined over a number field $K$ with $E(K)[2]\simeq \Zt$ that does not possess a cyclic 4-isogeny defined over $K$. Assuming conjecture $\Sha T_2(K)$ holds and either $d_2(E/K) \equiv 1 \pmod{2}$ or $E$ does not have constant 2-Selmer parity, then the number of twists $E^F$ of $E$ having rank one grows at least as fast as $\frac{X}{\log X}$.
\end{cor}

Other results similar to Corollaries \ref{rank0} and \ref{rank1} when $E(K)[2] = 0$ are due to Ono and Skinner when $K = \Q$ and to Rubin and Mazur for general $K$. Similar results when $E(K)[2]\simeq\mathbb{Z}/2\mathbb{Z} \times \mathbb{Z}/2\mathbb{Z}$ are due to Skorobogotav and Swinnerton-Dyer \cite{OS} \cite{MR}, \cite{SSD}.

\subsection{Layout}
We begin in Section \ref{selmergroups} by recalling the definition of the 2-Selmer group and presenting a technique developed by Mazur and Rubin to compare the 2-Selmer group of $E$ with that of $E^F$. Section \ref{phiselmer} develops similar machinery for the Selmer group associated with a 2-isogeny of $E$.  In Section \ref{characterization} we explore what it means for $E$ to have a cyclic 4-isogeny over $K$ and over $K(E[2])$. We prove Theorem \ref{equaltwist} in Section \ref{pfeqltwist}. Theorem \ref{noisogthm} is proved Sections \ref{goingdown} and \ref{goingup}. Lastly, in Section \ref{cyclic4isogeny} we develop techniques specific to curves with cyclic 4-isogenies defined over $K(E[2])$ and use them to prove Theorems \ref{balancedmost} and \ref{isogpairthm}.

The proofs of Lemmas \ref{shrinkpbal} and \ref{lemmaunbdbal} are simpler in the case where $E$ does not have a cyclic 4-isogeny defined over $K(E[2])$ than in the case where $E$ acquires a cyclic 4-isogeny over $K(E[2])$. We therefore include the proof for the former case in the main text and relegate the proof for latter case to Appendix \ref{firstapp}. Some of the results in Section \ref{cyclic4isogeny} require local calculations specific to curves that acquire a cyclic 4-isogeny over $K(E[2])$ and some of these calculations appear in Appendix \ref{loc4isog}.

\section{The 2-Selmer Group}\label{selmergroups}
\subsection{Background}
We begin by recalling the definition of the 2-Selmer group. If $E$ is an elliptic curve defined over a number field $K$, then $E(K)/2(K)$ maps into $H^1(K, E[2])$ via the Kummer map. The 2-Selmer group of $E$ is a subgroup of $H^1(K, E[2])$ that attempts to bound the part of $H^1(K, E[2])$ cut out by the image of $E(K)/2(K)$. We can map $E(K_v)/2(K_v)$ into $H^1(K_v, E[2])$ via the Kummer map for any completion $K_v$ of $K$, and the following diagram commutes for every place $v$ of $K$, where $\delta$ is the Kummer map.

\begin{center}\leavevmode
\begin{xy} \xymatrix{
E(K)/2E(K)  \ar[d] \ar[r]^{\delta} & H^1(K, E[2]) \ar[d]^{Res_v} \\
E(K_v)/2E(K_v)   \ar[r]^{\delta} & H^1(K_v, E[2]) }
\end{xy}\end{center}

For each place $v$ of $K$, we define a distinguished local subgroup $H^1_f(K_v, E[2]) \subset H^1(K_v, E[2])$ by $\text{Image} \left (\delta: E(K_v)/2E(K_v) \hookrightarrow H^1(K_v, E[2]) \right )$. We define the \textbf{2-Selmer group} of $E/K$, denoted $\Sel_2(E/K)$, by $$\Sel_2(E/K) = \ker \left ( H^1(K, E[2]) \xrightarrow{\sum res_v} \bigoplus_{v\text{ of } K} H^1(K_v, E[2])/H^1_f(K_v, E[2]) \right ).$$ That is, the 2-Selmer group is the group of cohomology class in $H^1(K, E[2])$ whose restrictions locally come from points of $E(K_v)$ in each completion $K_v$ of $K$.

The  $2$-Selmer group is a finite dimensional $\Ft$-vector space that sits inside the exact sequence of $\Ft$-vector spaces $$0\rightarrow E(K)/2E(K) \rightarrow \Sel_2(E/K) \rightarrow \Sha(E/K)[2] \rightarrow 0,$$ where $\Sha(E/K)$ is the Tate-Shafaravich group of $E$.

\begin{definition}
We define the \textbf{2-Selmer rank of $E$}, denoted $d_2(E/K)$, by $$d_2(E/K) =  \dimF   \Sel_2(E/K) - \dimF  E(K)[2].$$
\end{definition}

One of the ways in which we study the 2-Selmer group of $E$ is by studying the local conditions $H^1_f(K_v, E[2])$. The following lemma provides us with a way to do that.

\begin{lemma}\label{eval}
\begin{enumerate}[(i)]
\item If $v \nmid 2\infty$, then $\dimF  H^1_f(K_v, E[2]) = \dimF   E(K_v)[2]$
\item If $v \nmid 2\infty$ and $E$ has good reduction at $v$ then $$H^1_f(K_v, E[2]) \simeq E[2]/(Frob_v-1)E[2]$$ with the isomorphism given by evaluation of cocycles in $H^1_f(K_v, E[2])$ at the Frobenius automorphism $Frob_v$.
\end{enumerate}
\end{lemma}
\begin{proof}
This is Lemma 2.2 in \cite{MR}.
\end{proof}

We would like to examine the behavior of $\Sel_2(E/K)$ under the action of twisting by a quadratic extension.

\begin{definition}Let $E$ be given by $E:y^2 = x^3 + Ax^2 + Bx + C$ and $F/K$ be a quadratic extension given by $F=K(\sqrt{d})$.  The \textbf{quadratic twist} of $E$ by $F$ denoted $E^F$ is the elliptic curve given by the model $y^2=x^3 + dAx^2 + d^2Bx + d^3C$. \\
 \end{definition}
There is an isomorphism $E \rightarrow E^F$ given by $(x, y) \mapsto (dx, d^{3/2}y)$ defined over $F$. Restricted to $E[2]$, this map gives a canonical $G_K$ isomorphism $E[2] \rightarrow E^F[2]$. This allows us to view $H_f^1(K_v, E^F[2])$ as sitting inside $H^1(K_v, E[2])$. We will study $\Sel_2(E^F/K)$ by considering the relationship  between $H^1_f(K_v, E[2])$ and $H_f^1(K_v, E^F[2])$. This relationship is addressed by the following lemma due to Kramer.

Given a place $w$ of $F$ above a place $v$ of $K$, we get a norm map $E(F_w) \rightarrow E(K_v)$, the image of which we denote by $E_\mathbf{N}(K_v)$. 

\begin{lemma}\label{normintersection}
Viewing $H_f^1(K_v, E^F[2])$ as sitting inside $H^1(K_v, E[2])$, we have $$H_f^1(K_v, E[2]) \cap H^1_f(K_v, E^F[2]) =  E_\mathbf{N}(K_v)/2E(K_v)$$
\end{lemma}
\begin{proof}
This is Proposition 7 in \cite{KK} and Proposition 5.2 in \cite{MR2}. The proof in \cite{MR2} works even at places above $2$ and $\infty$.
\end{proof}

This equality gives rise to the following lemma:

\begin{lemma}\label{localequality} Let $E$ be an elliptic curve defined over $K$, $v$ a place of $K$, and $F/K$ be a quadratic extension. Then
\begin{enumerate}[(i)]
\item $H^1_f(K_v, E[2]) = H^1_f(K_v, E^F[2])$ if either $v$ splits in $F/K$ or $v$ is a prime where $E$ has good reduction that is unramified in $F/K$ 
\item $H^1_f(K_v, E[2]) \cap H^1_f(K_v, E^F[2]) = 0$ if $v \nmid 2\infty$, $E$ has good reduction at $v$, and $v$ is ramified in $F/K$.
\end{enumerate}
\end{lemma}
\begin{proof}
Part $(i)$ is Lemma 2.10 in \cite{MR} and part $(ii)$ is Lemma 2.11 in \cite{MR}.
\end{proof}

\begin{lemma}\label{twisttorsion}
Suppose $E$ has good reduction at a prime $v$ away from 2 and $F/K$ is a quadratic extension ramified at $v$. Then $E^F(K_v)$ contains no points of order 4. It follows that $H^1_f(K_v, E^F[2])$ is the image of $E^F(K_v)[2]$ under the Kummer map.

\end{lemma}
\begin{proof}
Since $E$ had good reduction at $v$, $v \nmid 2$, and $F/K$ is ramified at $v$, Tate's algorithm gives us that $E^F$ has reduction type $I_0^*$ at $v$. Since $E^F$ has additive reduction at $c$, we have an exact sequence $$0 \rightarrow E_1^F(K_v) \rightarrow E_0^F(K_v) \rightarrow k_v^+ \rightarrow 0,$$  where $E_0^F(K_v)$ is the group of points of non-singular reduction on $E^F(K_v)$, $E_1^F(K_v)$ is isomorphic to the formal group of $E^F(K_v)$, and $k_v$ is the residue field of $K$ at $v$. Since $k_v^+$ and $E_1^F(K_v)$ have no points of order 2, we get that $E_0^F(K_v)$ has no points of order 2. Because $E^F$ has reduction type $I_0^*$ at $v$, $E^F(K_v)/E_0^F(K_v) \subset \Zt \times \Zt$ (see Table 4.1 in section IV.9 of \cite{ATAEC}). It then follows that $E^F(K_v)$ has no points of order 4. As $E^F(K_v)$ is given by $E^F(K_v)[2^\infty] \times B$ for some profinite group $B$ of odd order, we get that  $H^1_f(K_v, E^F[2]) = E^F(K_v)[2^\infty]/2E^F(K_v)[2^\infty] = E^F(K_v)[2]$.
\end{proof}

\subsection{The Method of Rubin and Mazur}\label{methodsofrubingandmazur}

Mazur and Rubin developed a technique to use the comparison between the local conditions for $E$ and $E^F$  to compare the Selmer groups $\Sel_2(E/K)$ and $\Sel_2(E^F/K)$. We explain this technique here.

Let $E$ be an elliptic curve defined over a number field $K$ and $T$ a finite set of places of $K$. We define a localization map $$\loc_T:H^1(K, E[2]) \rightarrow \bigoplus_{v \in T} H^1(K_v, E[2])$$ as the sum of the restriction maps over the places $v$ in $T$. We define the strict and relaxed Selmer group, denoted $S_T$ and $S^T$ respectively, by $$S_T = \ker \left ( \loc_T:\Sel_2(E/K) \rightarrow \bigoplus_{v \in T} H^1(K_v, E[2]) \right) $$ and $$S^T = \ker \left ( \loc_T:H^1(K, E[2]) \rightarrow \bigoplus_{v \not \in T} H^1(K_v, E[2])/H^1_f(K_v, E[2]) \right ).$$

Lemma 3.2 in \cite{MR} shows that $\dimF  S^T - \dimF  S_T$ is given by $$\dimF  S^T - \dimF  S_T  = \sum_{v \in T} \dimF  H^1_f(K_v, E[2]).$$



The following theorem of Kramer provides an important relationship between the parities of $d_2(E/K)$ and $d_2(E^F/K)$.

\begin{theorem}[Kramer]\label{kramers} $$d_2(E/K) \equiv d_2(E^F/K) + \sum_{v \text{ of }K} \dimF  E(K_v)/E_\mathbf{N}(K_v)$$
\end{theorem}
\begin{proof}
This is Theorem 2.7 in \cite{MR}; also see Remark 2.8 there as well.
\end{proof}

The following proposition is the main ingredient in the method of Rubin and Mazur. We reproduce it here along with their proof.

\begin{proposition}[Proposition 3.3 in \cite{MR}]\label{MR3.3} Let $E$ be an elliptic curve defined over a number field $K$. Suppose $F/K$ is a quadratic extension such that all places above $2\Delta_E\infty$ split in $F/K$, where $\Delta_E$ is the discriminant of some model of $E$. Let $T$ be the set of (finite) primes $\mathfrak{p}$ of $K$ such that $F/K$ is ramified at $\mathfrak{p}$ and $E(K_\mathfrak{p})[2] \ne 0$. Set $V_T = \loc_T(\Sel_2(E/K)$.Then
$$d_2(E^F/K) = d_2(E/K)  - \dimF V_T + d$$ for some $d$ satisfying 
\begin{equation}\label{ineq}0 \le d \le \dimF \left ( \bigoplus_{\mathfrak{p} \in T}  H^1_f(K_\mathfrak{p}, E[2]) \right ) / V_T   \end{equation} and 
\begin{equation}\label{parity} d \equiv \dimF  \left ( \bigoplus_{\mathfrak{p} \in T}  H^1_f(K_\mathfrak{p}, E[2])\right ) /V_T \pmod{2}. \end{equation}
\end{proposition}
\begin{proof}
Let $V_T^F = \loc_T(\Sel_2(E^F/K))$.  Lemma \ref{localequality} gives us that $H^1_f(K_v, E^F[2]) = H^1_f(K_v, E[2])$ for all $v \not \in T$ and therefore $S_T \subset \Sel_2(E^F/K)$. This gives us that the sequences $$0 \rightarrow S_T \rightarrow \Sel_2(E/K) \rightarrow V_T \rightarrow 0$$ and $$ 0 \rightarrow S_T \rightarrow \Sel_2(E^F/K) \rightarrow V_T^F \rightarrow 0$$ are exact. We therefore get that $$d_2(E^F/K)  = d_2(E/K) - \dimF  V_T + \dimF  V_T^F.$$ We will let $d =  \dimF  V_T^F$ and show that it satisfies the conditions above.

Since $H^1_f(K_v, E[2]) \cap H^1(K_v, E^F[2])=0$ for all $v \in T$ by part $(ii)$ of Lemma \ref{localequality}, we have $$\dimF  V_T + \dimF  V_T^F = \dimF \Sel_2(E/K)/S_T + \dimF \Sel_2(E^F/K)/S_T$$ $$\le \dimF (S^T/S_T) = \sum_{v \in T} \dimF  H^1_f(K_v, E[2]).$$ The last equality follows from Lemma 3.2 in \cite{MR}. This gives us that $$\dimF  V_T^F \le \sum_{v \in T} \dimF  H^1_f(K_v, E[2]) -  \dimF  V_T = \dimF   \bigoplus_{\mathfrak{p} \in T}  H^1_f(K_\mathfrak{p}, E[2])/V_T,$$ giving us the first condition above.

To get the parity result, we will apply Theorem \ref{kramers}. Observe that Lemma \ref{normintersection} and part $(i)$ of Lemma \ref{localequality} tell us that $\dimF  E(K_v)/E_\mathbf{N}(K_v)=0$ for all $v \not \in T$ and that $\dimF  E(K_v)/E_\mathbf{N}(K_v)= \dimF  H^1_f(K_v, E[2])$ for all $v \in T$. We therefore have that $$\dimF  V_T + \dimF  V_T^F \equiv  \sum_{v \in T} \dimF  H^1_f(K_v, E[2]) \pmod{2}$$ so $$\dimF  V_T^F \equiv \dimF  \big(  \bigoplus_{\mathfrak{p} \in T}  H^1_f(K_\mathfrak{p}, E[2])/V_T  \big)\pmod{2}.$$
\end{proof}

\section{The $\phi$-Selmer Group}\label{phiselmer}
In the setting of Proposition \ref{MR3.3}, if $F/K$ is an extension such that $T$ contains a single prime $\mathfrak{p}$, then $d_2(E^F/K) - d_2(E/K)$ can often be read off of (\ref{ineq}) and (\ref{parity}) by considering the localization of $\Sel_2(E/K)$ at $\mathfrak{p}$. Mazur and Rubin developed a strategy for curves with $E(K)[2] =0$ in which one carefully constructs an extension $F/K$ so that $d_2(E^F/K) - d_2(E/K)$ can be read off of Proposition \ref{MR3.3}. This construction breaks down when $E(K)[2] \simeq \Zt$ because the action of $G_K$ on $E[2]$ is no longer irreducible. By utilizing the Selmer group associated to the 2-isogeny of $E$, we are able to develop a more complicated strategy that allows us to extend their results to the case when $E(K)[2] \simeq \Zt$.

\subsection{Background}
If $E$ is an elliptic curve defined over $K$ with $E(K)[2] \simeq \Zt$, then $E$ can be given by an integral model over $y^2 = x^3 + Ax^2 + Bx$ defined over $K$. The subgroup  $C = E(K)[2]$ is then generated by the point $P = (0,0)$.

Given this model, we are able to define a degree 2 isogeny $\phi:E \rightarrow E^\prime$ with kernel $C$, where $E^\prime$ is given by a model $y^2 = x^3 - 2Ax^2 + (A^2 - 4B)x$ and $\phi$ is given by $\phi(x, y) = \left( \left (\frac{x}{y} \right)^2, \frac{y(B-x^2)}{x^2} \right )$ for $(x, y) \not \in C$. If we define $C^\prime = \phi(E[2])$, then we get a short exact sequence of $G_K$ modules $$0 \rightarrow C \rightarrow E[2] \xrightarrow{\phi} C^\prime \rightarrow 0.$$ This short exact sequence gives rise to a long exact sequence of cohomology groups \begin{equation}\label{CE[2]Cprime}0 \rightarrow C \rightarrow E(K)[2] \xrightarrow{\phi} C^\prime \xrightarrow{\delta} H^1(K, C) \rightarrow H^1(K, E[2]) \xrightarrow{\phi} H^1(K, C^\prime) \rightarrow \ldots \end{equation} The map $\delta$ is given by $\delta(Q)(\sigma)= \sigma(R) - R$ where $R$ is any point on $E$ with $\phi(R) = Q$. 

In terms of the localization of the 2-Selmer group, we observe that if $c \in \Sel_2(E/K)$ comes from $H^1(K, C)$, then by part $(i)$ of Lemma \ref{eval}, the localization of $c$ at any prime of good reduction will lie in $C$. We define the \textbf{$\phi$-Selmer group of $E$}, denoted $\Sel_\phi(E/K)$, to capture the part of $\Sel_2(E/K)$ that comes from $H^1(K, C)$. We construct $\Sel_\phi(E/K)$ in a manner similar to the construction of $\Sel_2(E/K)$.

The following diagram commutes for every place $v$ of $K$.

\begin{center}\leavevmode
\begin{xy} \xymatrix{
E^\prime(K)/\phi(E(K))  \ar[d] \ar[r]^\delta & H^1(K, C) \ar[d]^{Res_v} \\
E^\prime(K_v)/\phi(E(K_v))   \ar[r]^\delta & H^1(K_v, C) }
\end{xy}\end{center}

We define distinguished local subgroups $H^1_\phi(K_v, C)\subset  H^1(K_v, C)$ for each place $v$ of $K$ as the image of $E^\prime(K_v)/\phi(E(K_v))$ under $\delta$ for each place $v$ of $K$ and define $\Sel_\phi(E/K)$ as $$\ker \left (  H^1(K, C) \xrightarrow{\oplus_{v} Res} \bigoplus_{v\text{ of }K} H^1(K_v, C)/ H^1_\phi(K_v, C) \right ).$$

The group $\Sel_\phi(E/K)$ is a finite dimensional $\Ftwo$-vector space and we denote its dimension $\dimF \Sel_\phi(E/K)$ by $d_\phi(E/K)$.

By identifying $C$ with $\mu_2$, we are able to explicitly identify $H^1(K, C)$ with $K^\times/(K^\times)^2$ and $H^1(K_v, C)$ with $K_v^\times/(K_v^\times)^2$ and consider $\Sel_\phi(E/K)$ as a subgroup of $K^\times/(K^\times)^2$. Explicitly, with the models as above, the map $E^\prime(K_v)/\phi(E(K_v))$ is given by $(x, y) \mapsto x(K_v^\times)^2$ for $(x, y) \ne (0,0)$ and $(0, 0) \mapsto \Delta_E(K_v^\times)^2$.

The local condition $H^1_\phi(K_v, C)$ is given by  $H^1_\phi(K_v, C) = H^1_u(K_v, C)$ for all finite places away from $2$ where $E$ has good reduction (see Lemma 4.1 in \cite{Cassels8}).

The isogeny $\phi$ on $E$ gives gives rise to the dual isogeny $\hat \phi$ on $E^\prime$ with kernel $C^\prime = \phi(E[2])$. Exchanging the roles of $(E, C, \phi)$ and $(E^\prime, C^\prime, \hat \phi)$ in the above defines the $\hat \phi$-Selmer group of $E^\prime$, $\Sel_\hatphi(E^\prime/K)$, as a subgroup of $H^1(K, C^\prime)$. The following theorem allows us to relate the $\phi$-Selmer group, the $\hat \phi$-Selmer group, and the 2-Selmer group.

\begin{theorem}\label{gss}The $\phi$-Selmer group, the $\hat \phi$-Selmer group, and the 2-Selmer group sit inside the exact sequence \begin{equation}0 \rightarrow E^\prime(K)[2]/\phi(E(K)[2]) \rightarrow \Sel_\phi(E/K) \rightarrow \Sel_2(E/K) \xrightarrow{\phi}\Sel_\hatphi(E^\prime/K).\end{equation}
\end{theorem}
\begin{proof}
This is a well known diagram chase. See Lemma 2 in \cite{FG} for example.
\end{proof}

\subsection{Duality Between $H^1_\phi(E/K)$ and $H^1_\hatphi(E^\prime/K)$}
A second relationship between the $\phi$-Selmer group and the $\hat \phi$-Selmer group arises from a duality between their respective local conditions.

\begin{proposition}\label{Re4.7} The sequence \begin{equation}\label{ss} 0 \rightarrow C^\prime/\phi \left ( E(K_v)[2]  \right ) \rightarrow H^1_\phi(K_v, C) \rightarrow H^1_f(K_v, E[2]) \xrightarrow{\phi} H^1_{\hat \phi}(K_v, C^\prime) \rightarrow 0\end{equation} sitting inside sequence (\ref{CE[2]Cprime}) is exact.
\end{proposition}
\begin{proof}
This is a well-known result. See Remark X.4.7 in \cite{AEC} for example.
\end{proof}

 \begin{lemma}[Local Duality]\label{localduality}
For each place $v$ of $K$ there is a local Tate pairing $H^1(K_v, C) \times H^1(K_v, C^\prime) \rightarrow \{\pm 1 \}$ induced by a pairing $[ \text{   }, \text{   } ] :C \times C^\prime  \rightarrow \{\pm 1 \}$ given by $[ Q , \tilde R ]  =  \langle Q, R \rangle$ where $\langle Q, R \rangle$ is the Weil pairing and $R$ is any pre-image of $\tilde R$ under $\phi$. The subgroups defining the local conditions $H^1_\phi(K_v, C)$ and $H^1_{\hat \phi}(K_v, C^\prime)$ are orthogonal complements under this pairing.
\end{lemma}
\begin{proof}
Orthogonality  is equation (7.15) and the immediately preceding comment in \cite{Cassels8}. Combining this orthogonality with exact sequence (\ref{ss}) gives that  $H^1_\phi(K_v, C)$ and $H^1_{\hat \phi}(K_v, C^\prime)$ are orthogonal complements. 
\end{proof}

\begin{definition} The ratio $$\mathcal{T}(E/E^\prime) = \frac{\big | \Sel_\phi(E/K) \big |}{\big |\Sel_{\hat \phi}(E^\prime/K)\big |}$$ is called the \textbf{Tamagawa ratio} of $E$. \end{definition}

Because of the duality  in Lemma \ref{localduality}, the Tamagawa ratio can be computed using a local product formula.

\begin{theorem}[Cassels]\label{prodform2}
The Tamagawa ratio $\mathcal{T}(E/E^\prime)$ is given by $$\mathcal{T}(E/E^\prime) = \prod_{v\text{ of } K}\frac{\left | H^1_\phi(K_v, C)\right |}{2}.$$
\end{theorem}
\begin{proof}
This is a combination of Theorem 1.1 and equations (1.22) and (3.4) in \cite{Cassels8}.
\end{proof}

We further have the following parity condition.

\begin{theorem}[Dokchitser, Dokchitset]\label{tparity}
$$d_2(E/K) \equiv \ord_2 \T(E/E^\prime) \pmod{2}.$$
\end{theorem}
\begin{proof}
This is Corollary 5.8 in \cite{DD2}.
\end{proof}

\subsection{Relationship Between $H^1_\phi(K_v, C)$ and $H^1_\phi(K_v, C^F)$}

As noted earlier, if $F/K$ is a quadratic extension, then there is a canonical $G_K$-isomorphsim $E[2] \rightarrow E^F[2]$. If $C$ is a subgroup of $E(K)$ of order 2, then we denote the image of $C$ in $E^F(K)[2]$ by $C^F$. As the map $C^F \rightarrow C$ is $G_K$ invariant, we can view $H^1_\phi(K_v, C^F)$ as a subgroup of $H^1(K_v, C)$ and $\Sel_\phi(E^F/K)$ as a subgroup of $H^1(K, C)$. This can be thought of as identifying both $H^1(K, C)$ and $H^1(K, C^F)$ with $K^\times/(K^\times)^2$ and both $H^1(K_v, C)$ and $H^1(K_v, C^F)$ with $K_v^\times/(K_v^\times)^2$.

The following four lemmas are an analogue of Lemma  \ref{localequality} that allow us to compare $H^1_\phi(K_v, C)$ and $H^1_\phi(K_v, C^F)$.

\begin{lemma}\label{localconds}
Suppose $v$ is a prime away from 2 where $E$ has good reduction and $v$ is ramified in $F/K$. Then $H^1_\phi(K_v, C^F) =  E^{\prime F}(K_v)[2]/\phi( E^F(K_v)[2])$.
\end{lemma}
\begin{proof}
Since $v$ is a prime away from 2, $E^F(K_v)$ is given by $E^F(K_v)[2^\infty] \times B$ and $E^{\prime F}(K_v)$ is given by $E^{\prime F}(K_v)[2^\infty] \times B^\prime$ where $B$ and $B^\prime$ are profinite abelian groups of odd order. Since $\phi$ has degree 2, we therefore get that $E^{\prime F}(K_v)/\phi( E^F(K_v) ) = E^{\prime F}(K_v)[2^\infty]/\phi( E^F(K_v)[2^\infty])$. By Lemma \ref{twisttorsion}, neither $E^F(K_v)$ nor $E^{\prime F}(K_v)$ have any points of order 4, yielding the result.
\end{proof}

\begin{lemma}[Criteria for equality of local conditions after twist]\label{localeq}

If any of the following conditions hold:
\begin{enumerate}[(i)]
\item $v$ splits in $F/K$
\item $v$ is a prime away from 2 where $E$ has good reduction and $v$ is unramified in $F/K$
\item $v$ is a prime away from 2 where $E$ has good reduction, $v$ is ramified in $F/K$, and $E(K_v)[2] \simeq \Zt \simeq E^\prime(K_v)[2]$
\end{enumerate}
then $H^1_\phi(K_v, C) = H^1_\phi(K_v, C^F)$ and  $H^1_{\hat \phi}(K_v, C^\prime) = H^1_{\hat \phi}(K_v, C^{\prime F})$.
\end{lemma}

\begin{proof} \text{ }\begin{enumerate}[(i)]
\item In this case $E \simeq E^F$ over $K_v$.
\item If $v$ is a prime away from 2 where $E$ has good reduction and $v$ is unramified in $F/K$, then $E^F$ also has good reduction at $v$. It then follows that both $H^1_\phi(K_v, C^F)$ and $H^1_\phi(K_v, C)$ are equal to $H^1_u(K_v, C)$ and both $H^1_{\hat \phi}(K_v, C^\prime)$ and $H^1_{\hat \phi}(K_v, C^{\prime F})$ are equal to $H^1_u(K_v, C^\prime)$.
\item By Lemma \ref{localconds}, $H^1_\phi(K_v, C^F) =  E^{\prime F}(K_v)[2]/\phi( E^F(K_v)[2])$. Since $E(K_v)[2] \simeq \Zt \simeq E^\prime(K_v)[2]$, we therefore get that $H^1_\phi(K_v, C^F)$ is given by the image of $C^{\prime F}$ in $H^1(K_v, C)$. Since the image of $C^{\prime F}$ is generated by $\Delta_E$, $E$ has good reduction at $v$, and $\Delta_E\not \in (K_v^\times)^2$, we get that the image of $C^{\prime F}$ is $H^1_u(K_v, C)$ and therefore that $H^1_\phi(K_v, C) = H^1_\phi(K_v, C^F)$. Exchanging the roles of $E$ and $E^\prime$ then gives the result.
\end{enumerate}\end{proof}

\begin{lemma}\label{twistbp} Suppose $E$ has good reduction at a prime $v$ away from 2 and $F/K$ is a quadratic extension ramified at $v$. If $E(K_v)[2] \simeq \Zt \times \Zt$ and $E^\prime(K_v)[2] \simeq \Zt \times \Zt$, then both of $H^1_\phi(K_v, C^F)$ and $H^1_{\hat \phi}(K_v, C^{\prime F})$ have $\mathbb{F}_2$ dimension 1 and both $H^1_\phi(K_v, C) \cap H^1_\phi(K_v, C^F) = 0$ and $H^1_{\hat \phi}(K_v, C^\prime) \cap H^1_{\hat \phi}(K_v, C^{\prime F}) = 0$. Further, if $F_1/K_v$ and $F_2/K_v$ are different quadratic extensions of $K_v$ both of which are ramified, then $H^1_\phi(K_v, C^{F_1}) \cap H^1_\phi(K_v, C^{F_2}) = 0$.
\end{lemma}
\begin{proof}
The fact that $H^1_\phi(K_v, C^F)$ and $H^1_{\hat \phi}(K_v, C^{\prime F})$ both have $\mathbb{F}_2$ dimension 1 is an immediate consequence Lemma \ref{localconds}. Further, as a subgroup of $K_v^\times/(K_v^\times)^2$, $H^1_\phi(K_v, C^F)$ is generated by the image of $Q^\prime$, where $Q^\prime \in E^{\prime F}[2] - C^{\prime F}$. If $F/K$ is locally given by $F_w = K_v(\sqrt{\pi})$, then the point $Q^\prime$ on $E^{F_w}$ is given by $(u\pi, 0)$, where $(u, 0) \in E^\prime(K_v)[2]$. Since $H^1_\phi(K_v, C) = H^1_u(K_v, C)$, it follows that $\ord_v u$ is even. Since $v \nmid 2$ and $F_w$ is ramified at $v$, we get that $\ord_v \pi$ is odd. Since $Q^\prime$ maps to $u\pi(K_v^\times)^2$,  $H^1_\phi(K_v, C^F)$ is a ramified subgroup of $H^1(K_v, C)$ and therefore disjoint from $H^1_\phi(K_v, C) = H^1_u(K_v, C)$. Exchanging the roles of $E$ and $E^\prime$ completes the first part of the result.

Since $F_1 \ne F_2$, we may write $F_1 = K_v(\sqrt{\pi})$ and $F_2 = K_v(\sqrt{w\pi})$, where $w \in O_{K_v}^\times - (O_{K_v}^\times)^2$. By the above, we get that $H^1_\phi(K_v, C^{F_1})$ is generated by $u\pi(K_v^\times)^2$ and  $H^1_\phi(K_v, C^{F_2})$ is generated by $uw\pi(K_v^\times)^2$. Since $w \not \in (K_v^\times)^2$, the result follows.
\end{proof}

\begin{lemma} \label{kill local} Suppose $E$ has good reduction at a prime $v$ away from 2 and $F/K$ is a quadratic extension ramified at $v$. If $E(K_v)[2] \simeq \Zt \times \Zt$, and $E^\prime(K_v)[2] \simeq \Zt$, then $H^1_\phi(K_v, C^F) = 0$ and  $H^1_\hatphi(K_v, C^{\prime F})=H^1(K_v, C)$. 
\end{lemma}
\begin{proof}
This follows immediately from Lemma \ref{localconds}.
\end{proof}

\section{Characterization of curves with $E(K) \simeq \mathbb{Z}/2\mathbb{Z}$}\label{characterization}

For the remainder of this paper, we will assume that $E(K)[2] \simeq \Zt$, $M = K(E[2])$, and $M^\prime = K(E^\prime[2])$.

The results obtained in this paper are dependent on whether $E$ has a cyclic isogeny defined over $K$ or over $K(E[2])$. The following result characterizes such curves.

\begin{lemma}\label{charac}
Let $E$ be an elliptic curve defined over $K$ with $E(K) \simeq \mathbb{Z}/2\mathbb{Z}$.
\begin{enumerate}[(i)]
\item  $E$ has a cyclic 4-isogeny defined over $K$ if and only if $M^\prime = K(E^\prime[2]) = K$. 
\item  $E$ does not have a cyclic 4-isogeny defined over $K$ but acquires a cyclic 4-isogeny over $M=K(E[2])$ if and only if $M = M^\prime$
\end{enumerate}
\end{lemma}
\begin{proof}
We begin with the following two observations. First, if $R \in E(\overline{K})$ with $2R = P$, then $\phi(R) \in E^\prime[2] - C^\prime$. Second, if $Q^\prime \in E^\prime[2] - C^\prime$ and $R \in E(\overline{K})$ with $\phi(R) = Q^\prime$, then $2R = \hatphi(Q^\prime) = P$.

Suppose that $E$ has a cyclic 4-isogeny defined over $K.$ Let $H = \langle R \rangle$ be a cyclic subgroup of order 4 fixed by $G_K$. Since $2R \in E[2]$ and $2R$ is fixed by $G_K$, we get that $2R = P$. We will show that $\phi(R) \in E^\prime(K)$. Let $\sigma \in G_K$. As $\phi$ is defined over $K$, $\sigma(\phi(R)) = \phi(\sigma(R))$. However, since $\sigma(R) \in R + C$ for all $\sigma \in G_K$, it follows that $\sigma(\phi(R)) = \phi(R)$. Since $\phi(R) \in E^\prime[2] - C$ and $\phi(R)$ is defined over $K$, we therefore get that $M^\prime = K$.

Now suppose that $E$ does not have a cyclic 4-isogeny defined over $K$. Let $Q^\prime \in E^\prime[2] - C^\prime$ and take $R \in E(\overline{K})$ with $\phi(R) = Q^\prime$. Since $E$ does not have a cyclic 4-isogeny defined over $K$, there is some $\sigma \in G_K$ such that $\sigma(R) \not \in \langle R \rangle = \{ 0 , R, P, R+P \}$. We therefore have that $\sigma(\phi(R)) = \phi(\sigma(R)) \ne \phi(R).$ Since $\phi(R) \in E^\prime[2] - C^\prime$, it follows that $M^\prime \not \subset K$. This proves $(i)$. Further, replacing $K$ by $M$ shows that if $E$ does not have a cyclic 4-isogeny defined over $M$, then $M^\prime \not \subset M$.

Now suppose that $E$ does not have a cyclic 4-isogeny defined over $K$ but acquires a cyclic 4-isogeny over $M=K(E[2])$. That is, there is a point $R \in E(\overline{K})$ of order 4 such that $\langle R \rangle$ is fixed by $G_M$.  If $R$ is a point of order 4 on $E$ such that $2R \ne P$, then $\phi(R)$ is a point of order 4 such that $2\phi(R) \in C^\prime$. Since $M^\prime \not \subset K$, $M^\prime$ being contained in $M$ is equivalent to $M$ being contained in $M^\prime$. By passing to $E^\prime$ if necessary, we can therefore assume that $2R = P$. We therefore get that  $\sigma(R) \in R + C$ for all $\sigma \in G_M$ and that $\phi(\sigma(R)) = \phi(R)$ for all $\sigma \in G_M$. As $\phi(\sigma(R)) = \sigma(\phi(R))$, we get that $\phi(R) \in E^\prime(M)$. Since $\phi(R) \in  E^\prime[2] - C^\prime$, this gives $M^\prime \subset M$. As $E$ did not have a cyclic 4-isogeny defined over $K$, we have $K \genfrac{}{}{0pt}{}{\subset}{\ne} M^\prime \subset M$, showing that $M = M^\prime$, completing the proof of $(ii)$.
\end{proof}

The following corollary follows immediately. 
\begin{corollary}\label{characc}
With $E$ as in Lemma \ref{charac}:
If $E$ possesses a cyclic 4-isogeny defined over $M$ but not one defined over $K$, then $\dimF E(K_v)[2] = \dimF E^\prime(K_v)[2]$ for all places $v$ of $K$.
\end{corollary}

\section{Proof of Theorem \ref{equaltwist}}\label{pfeqltwist}
For the remainder of this paper, we will be assuming that $E(K)[2] \simeq \Zt$ and that $E$ does not have a cyclic 4-isogeny defined over $K$. For simplicity, we will let $C = E(K)[2]$ be generated by a point $P$ and fix a second point $Q$ of order 2 so that $E[2] = \langle P, Q \rangle$.


The proofs of our main theorems are based on constructing an extension $F/K$ ramified at a small number of places in such a fashion that $d_2(E^F/K) - d_2(E/K)$ may be read off of (\ref{ineq}) and (\ref{parity}).

\begin{lemma}\label{pick 3} Define $S$ to be the image of $H^1(K, C)$ in $H^1(K, E[2])$ under the map in (\ref{CE[2]Cprime}) and let $V \subseteq H^1(K, E[2])$ be a dimension $r$ subspace of $H^1(K, E[2])$ such that $V \cap S = 0$. Then there exist $\sigma_1, ..., \sigma_r \in G_K$ such that $\sigma_i \big |_M \ne 1$ and $\loc(V)$ has dimension $r$ where the localization map $\loc:H^1(K, E[2]) \rightarrow \bigoplus_{i=1}^r E[2]/C$ is defined by $c \mapsto \left (c(\sigma_1), ..., c(\sigma_r) \right)$.
\end{lemma}
\begin{proof} Let $\tilde Q$ be the image of $Q$ in $E[2]/C$. Take $\tilde V \subseteq H^1(K, C^\prime)$ to be the image of $V$ in $H^1(K, C^\prime) = \text{Hom}(G_K, C^\prime)$ and observe that $\tilde V$ has dimension $r$ since $V \cap \text{Image}\left ( H^1(K, C) \right ) = 0$. Take a basis $\left \{ c_1, ..., c_r\right \}$ for $\tilde V$. Since the image of the dual map $G_K \rightarrow \text{Hom}(\tilde V, C^\prime)$ is surjective and has dimension $r$ as well, we can find $\tau_1, ..., \tau_r \in G_K$ such that $$c_i(\tau_j)=\left \{ \begin{array}{cl} \tilde Q  & \text{if  } i = j \\ 0 & \text{if } i\ne j \end{array} \right . $$

Suppose $\tau_i \in G_M$ for all $i$. Then the map $G_M \rightarrow \text{Hom}(\tilde V, C^\prime)$ must be surjective as well. So taking any $\tau \in G_K - G_M$, we can therefore find $\gamma_1, ..., \gamma_r \in G_M$ such that $$c_i(\gamma_j)=\left \{ \begin{array}{cl} \tilde Q + c_j(\tau) & \text{if  } i = j \\ c_j(\tau) & \text{if } i\ne j \end{array} \right .$$ 

Thus, $$c_j(\gamma_i\tau)= c_j(\gamma_i) + c_j(\tau) = \left \{ \begin{array}{cl} \tilde Q& \text{if  } i = j \\ 0 & \text{if } i\ne j \end{array} \right . $$ Taking $\sigma_i = \gamma_i\tau$ then gives the result.

Otherwise, we have $\tau_k \notin G_M$ for some $k$. Define  $\sigma_1, ..., \sigma_r$ by $$\sigma_i = \left \{ \begin{array}{cl}\tau_i & \text{if  } \tau_i \notin G_M \\ \tau_i\tau_k & \text{if }  \tau_i \in G_M  \end{array} \right . $$

We then get that $$c(\sigma_i)=  0 \text{   } \forall  i \Leftrightarrow  c(\tau_i) =0  \text{   }\forall  i$$ which means that the localization map is injective on $V$ giving us the result.
\end{proof}

The following proposition is a consequence of Lemma \ref{pick 3} that will be used to prove Theorem \ref{equaltwist}.

\begin{proposition}\label{earlier} Let $E$ be an elliptic curve over a number field $K$ with $E(K)[2]\simeq \mathbb{Z}/2\mathbb{Z}$ that does not have a cyclic 4-isogeny defined over $K$. If $r = d_2(E/K)$, then $$N_r(E, X) \gg \frac{X}{\log X}.$$
\end{proposition}
\begin{proof} Pick $R \in E(\overline K)$ such that $2R = P$. Since $E$ does not have a cyclic 4-isogeny over $K$, there exists some $\tau \in G_K$ such that $\tau(R) \not \in \langle R \rangle = \{ 0, R, P, R+P\}$. Letting$c_0$ be the co-cycle $c_0 : G_K \rightarrow E[2]$ given by $c_0(\sigma) = \sigma(R) - R$ for $\sigma \in G_K$, we see that  $c_0(\tau) \not \in C$ and therefore that the element of $\Sel_2(E/K)$ coming from $P$ does not come from $H^1(K, C)$. By Lemma \ref{pick 3}, we may therefore find $\gamma \in G_K$ such that $\gamma \big |_M \ne 1$ such that $c_0(\gamma)$ is non-trivial in $E[2]/C$.

Let $N$ be a finite Galois extension of $K$ containing $MK(8\Delta_E\infty)$ such that the restriction of $\Sel_2(E/K)$ to $N$ is zero, where $K(8\Delta_E\infty)$ is the ray class field modulo $8\Delta_E\infty$. Let $\mathfrak{p}_1, \mathfrak{p}_2$ be primes of $K$ where $E$ has good reduction, not dividing 2, where $\text{Frob}_{\mathfrak{p}_1}$ in Gal$(N/K)$ is the conjugacy class of $\gamma \big |_N$ and $\text{Frob}_{\mathfrak{p}_2}$ in Gal$(N/K)$ is the conjugacy class of $\gamma^{-1} \big |_N$. Since $\gamma\gamma^{-1} \big |_{K(8\Delta_E\infty)}=1$, $\mathfrak{p}_1\mathfrak{p}_2$ has a totally positive generator $\pi \equiv 1 \text{  (mod } 8\Delta_E)$. Letting $F=K(\sqrt \pi)$, we get that all places dividing $2\Delta_E\infty$ split in $F/K$, and $\mathfrak{p}_1$, $\mathfrak{p}_2$ are the only primes that ramifiy in $F/K$.

We then apply Proposition \ref{MR3.3} with $T=\{\mathfrak{p}_1, \mathfrak{p}_2\}$. Since $E$ has good reduction at $\mathfrak{p}_1$ and $\mathfrak{p}_2$, it follows from Lemma \ref{eval} that $H^1_f(K_{\mathfrak{p}_1}, E[2]) = E[2]/(\gamma-1)E[2] = E[2]/C$, $H^1_f(K_{\mathfrak{p}_2}, E[2]) = E[2]/(\gamma^{-1}-1)E[2] = E[2]/C$ and that the localization $\loc_T:\Sel_2(E/K) \rightarrow H^1_f(K_{\mathfrak{p}_1}, E[2]) \oplus H^1_f(K_{\mathfrak{p}_2}, E[2])$ is given by $c \mapsto \left (c(\gamma), c(\gamma^{-1})\right )$.

If $c \in H^1(K, E[2])$, then we have $0 = c(1) =c(\gamma\gamma^{-1})=c(\gamma) + c(\gamma^{-1})$, giving that $c(\gamma) = c(\gamma^{-1})$. Letting $V_T$ be the image of $\Sel_2(E/K)$ under the localization map, we therefore get that $V_T$ is contained in the one dimensional diagonal subspace of $H^1_f(K_{\mathfrak{p}_1}, E[2]) \oplus H^1_f(K_{\mathfrak{p}_2}, E[2])$. Since $c_0(\gamma)$ is non-trivial in $E[2]/C$, we get that $V_T$ is in fact equal to the diagonal subspace in $H^1_f(K_{\mathfrak{p}_1}, E[2]) \oplus H^1_f(K_{\mathfrak{p}_2}, E[2])$.

By Proposition \ref{MR3.3}, we then get that $d_2(E^F/K) = d_2(E/K)$.

Next, observe that if we fix $\mathfrak{p}_1$, we still have complete liberty in choosing $\mathfrak{p}_2$ subject to the conditions on its Frobenius. The Chebotarev density theorem then gives the result. 
\end{proof}

\begin{proof}[Proof of Theorem \ref{equaltwist}] 
We follow the proof of Theorem 1.4 in \cite{MR}. Suppose $d_2(E^F/K) = r$. Every twist $(E^F)^{L^\prime}$ of $E^F$ is also a twist $E^L$ of $E$ and $$\mathfrak{f}(L/K) | \mathfrak{f}(F/K)\mathfrak{f}(L^\prime/K) $$ so $N_r(E, X) \ge N_r\left (E^F, \frac{X}{\mathbf{N}_{K/\mathbb{Q} } \mathfrak{f}(F/K)}\right )$. 

Since possession of a cyclic 4-isogeny defined over $K$ is fixed under twisting, $E^F$ does not possess a cyclic isogeny of degree 4 defined over $K$. The result then follows from applying Proposition \ref{earlier} to $E^F$.
\end{proof}

\section{Twisting to Decrease Selmer Rank}\label{goingdown}

As shown in the proof of Proposition \ref{earlier}, Lemma \ref{pick 3} allows us to effectively use Lemma \ref{MR3.3} as long as the image of $\Sel_\phi(E/K)$ in $\Sel_2(E/K)$ is not too large. The following proposition provides another example of this.

\begin{proposition} \label{shrink2}  
Let $E$ be an elliptic curve defined over a number field $K$ such that $E(K)[2] \simeq \mathbb{Z}/2\mathbb{Z}$ and $E$ does not have a cyclic 4-isogeny defined over $K$. If $\dimF \Sel_\phi(E/K) \le \dimF \Sel_2(E/K) -2$, then $E$ has a twist $E^F$ such that $d_2(E^F/K) = d_2(E/K)-2$.
\end{proposition}

\begin{proof} Let $S$ be the image of $\Sel_\phi(E/K)$ in $\Sel_2(E/K)$. Since $E^\prime(K)[2]) \simeq \Zt$ by  Lemma \ref{charac}, $\Sel_\phi(E/K)$ maps 2-1 into $\Sel_2(E/K)$ by Theorem \ref{gss}. As $\dimF \Sel_\phi(E/K) \le \dimF \Sel_2(E/K) -2$, we can therefore find a 3-dimensional $\Ftwo$-subspace $V \subset \Sel_2(E/K)$ such that $V \cap S = 0$.  Pick  $\gamma_1, \gamma_2, \gamma_3 \in G_K$ as in Lemma \ref{pick 3} and set $\gamma_4 = (\gamma_1\gamma_2\gamma_3)^{-1}$. 

Let $N$ be a finite Galois extension of $K$ containing $MK(8\Delta_E\infty)$ such that the restriction of $\Sel_2(E/K)$ to $N$ is zero. Choose 4 primes $\mathfrak{p}_1, \mathfrak{p}_2, \mathfrak{p}_3, \text{ and } \mathfrak{p}_4$ not dividing 2 such that $E$ has good reduction and each $\mathfrak{p}_i$ and $\text{Frob}_{\mathfrak{p_i}} \big |_N = \gamma_i \big |_N$ for $i = 1, ..., 4$. Since $\gamma_1\gamma_2\gamma_3\gamma_4 \big |_{K(8\Delta_E\infty)}=1$, the ideal $\mathfrak{p}_1\mathfrak{p}_2\mathfrak{p}_3\mathfrak{p}_4$ has a totally positive generator $\pi$ with $\pi \equiv 1\text{  }\pmod{8\Delta_E\infty}$. Setting $F=K(\sqrt{\pi})$, it follows that all places dividing $2\Delta_E\infty$ split in $F/K$, and $\mathfrak{p}_1, \mathfrak{p}_2, \mathfrak{p}_3, \text{ and } \mathfrak{p}_4$ are the only primes that ramify in $F/K$.

We can therefore apply Proposition \ref{MR3.3} with $T = \{ \mathfrak{p}_1, \mathfrak{p}_2, \mathfrak{p}_3,\mathfrak{p}_4\}$. Since none of $\gamma_1, \gamma_2, \gamma_3$ were trivial on $M$, it follows that $\gamma_4 \big |_M \ne 1$ as well. Since $E$ has good reduction at each of the $\mathfrak{p}_i$, it follows from part $(ii)$ of Lemma \ref{eval} that $H^1_f(K_{\mathfrak{p}_i}, E[2]) = E[2]/(\gamma_i-1)E[2] = E[2]/C$ and that the localization map $$\text{loc}_T:\Sel_2(E/K) \rightarrow H^1_f(K_{\mathfrak{p}_1}, E[2]) \oplus H^1_f(K_{\mathfrak{p}_2}, E[2]) \oplus H^1_f(K_{\mathfrak{p}_3}, E[2])  \oplus H^1_f(K_{\mathfrak{p}_4}, E[2])$$ is given by $$c \mapsto \left (c(\gamma_1), c(\gamma_2), c(\gamma_3),  c(\gamma_4) \right ).$$

Let $V_T$ be the image of $\Sel_2(E/K)$ under the localization map. By choosing $\gamma_1, \gamma_2, \text{ and }  \gamma_3$ in accordance with Lemma \ref{pick 3}, we ensure that  $V_T$ has dimension at least 3. If $c \in H^1(K, E[2])$, then $0 = c(1) = c(\gamma_1\gamma_2\gamma_3\gamma_4) = c(\gamma_1) + c(\gamma_2) + c(\gamma_3) +  c(\gamma_4)$, showing that $(0, 0, 0, Q) \not \in V_T$ and there that $V_T$ has dimension exactly 3. Proposition \ref{MR3.3}  therefore gives that $d_2(E^F/K) = d_2(E/K) - 2$.
\end{proof}

The key observation to using Proposition \ref{shrink2} to obtain twists of $E$ with reduced 2-Selmer rank is realizing that the size of the image of $\Sel_\phi(E/K)$ in $\Sel_2(E/K)$ can be controlled effectively.
 
 \begin{lemma}\label{shrinkpbal}
Let $E$ be an elliptic curve defined over a number field $K$ with $E(K)[2] \simeq \mathbb{Z}/2\mathbb{Z}$ that does not possess a cyclic 4-isogeny defined over $K$. If both $d_\phi(E/K) > 1$ and $d_{\hat \phi}(E^\prime/K) > 1$, then $E$ has a quadratic twist $E^F$ such that 
\begin{enumerate}[(a)]
\item $d_\phi(E^F/K) < d_\phi(E/K) $, 
\item $d_2(E^F/K) = d_2(E/K)$ or $d_2(E^F/K) = d_2(E/K)-2$, and
\item $\T(E^F/E^{\prime F}) \le \T(E/E^\prime)$
\end{enumerate}
\end{lemma}

\begin{proof}
We will present the proof for the case when $E$ does not have a cyclic 4-isogeny defined over $M$ here and refer to the reader to Appendix \ref{firstapp} for the proof when $E$ acquires a cyclic 4-isogeny over $M$. The requirement that $d_{\hat \phi}(E^\prime/K) > 1$ is required only in that latter case.

Suppose $E$ does not have a cyclic 4-isogeny defined over $M$. Pick $R \in E(\overline K)$ such that $2R = P$. Since $E$ does not have a cyclic 4-isogeny over $M$, there exists some $\tau \in G_M$ such that $\tau(R) \not \in \langle R \rangle = \{ 0, R, P, R+P\}$. Letting $c_0$ be the co-cycle $c_0 : G_K \rightarrow E[2]$ given by $c_0(\sigma) = \sigma(R) - R$ for $\sigma \in G_K$, we see that  $c_0(\tau) \not \in C$. Observing that $$\phi(c_0(\sigma)) = \phi(\sigma(R) - R) = \phi(\sigma(R)) -\phi(R) = \sigma(\phi(R)) - \phi(R)$$ for $\sigma \in G_K$, we see that $\tau(\phi(R)) \ne \phi(R)$ since $c_0(\tau) \not \in C = \ker \phi$. Recalling that $\phi(R) \in E^\prime[2]$ from the proof of Lemma \ref{charac}, this shows that $\tau$ does not fix $E^\prime[2]$ and therefore that $\tau \big |_{M^\prime} \ne 1$.

Since  $d_\phi(E/K) > 1$, we can find $b \in \Sel_\phi(E/K)$ such that $b$ does not come from $C^\prime$. Since $b$ does not come from $C^\prime$, taking $N_{b} = \overline{K}^{\text{Ker}b}$, we get that $N_b \ne M$. We may therefore find $\gamma \in G_K$ such that $\gamma \big |_{MM^\prime} = \tau$ and $\gamma \big |_{N_b} \ne 1$.

Let $N$ be a finite Galois extension of $K$ containing $MK(8\Delta_E\infty)$ such that the restrictions of $\Sel_2(E/K)$ and  $\Sel_\phi(E/K)$ to $N$ are zero. Let $\mathfrak{p}_1$ and $\mathfrak{p}_2$ be primes of $K$ away from $2$ where $E$ has good reduction such that $\text{Frob}_{\mathfrak{p}_1}$ in Gal$(N/K)$ is the conjugacy class of $\gamma \big |_N$ and $\text{Frob}_{\mathfrak{p}_2}$ in Gal$(N/K)$ is the conjugacy class of $\gamma^{-1} \big |_N$. Since $\gamma \gamma^{-1} \big |_{K(8\Delta_E\infty)}=1$, $\mathfrak{p}_1\mathfrak{p}_2$ has a totally positive generator $\pi \equiv 1 \text{  (mod } 8\Delta_E)$. Letting $F=K(\sqrt \pi)$, we get that all places dividing $2\Delta_E\infty$ split in $F/K$, and that $\mathfrak{p}_1$ and $\mathfrak{p}_2$ are the only primes that ramify in $F/K$.

We now apply Proposition \ref{MR3.3} with $T=\{\mathfrak{p}_1, \mathfrak{p}_2\}$. Since $E$ has good reduction at $\mathfrak{p}_1$ and $\mathfrak{p}_2$, it  follows that $H^1_f(K_{\mathfrak{p}_1}, E[2]) = E[2]/(\gamma-1)E[2] = E[2]$, $H^1_f(K_{\mathfrak{p}_2}, E[2]) = E[2]/(\gamma^{-1}-1)E[2] = E[2]$ and that the localization map $\loc_T:\Sel_2(E/K) \rightarrow H^1_f(K_{\mathfrak{p}_1}, E[2]) \oplus H^1_f(K_{\mathfrak{p}_2}, E[2])$ is given by $c \mapsto \left (c(\gamma), c(\gamma^{-1})\right )$. If $c \in H^1(K, E[2])$, the we have $0 = c(1) =c(\gamma\gamma^{-1})=c(\gamma) +  c(\gamma^{-1})$ giving that $c(\gamma) = c(\gamma^{-1})$. Letting $V_T$ be the image $\Sel_2(E/K)$ under the localization map, we therefore get that $V_T$ is contained in the two-dimensional diagonal subspace of $E[2] \times E[2]$.

As $b \in \Sel_\phi(E/K)$, we may view it as an element of $\Sel_2(E/K)$. Taking any co-cycle $\hat b$ representing $b$, we get that $\hat b(\gamma) \ne 0$ since $\gamma \big |_{N_{b}} \ne 1$ and that $\hat b(\gamma) \in C$ since $b \in H^1(K, C)$. By design, $c_0(\gamma) \not \in C$. As $c_0$ represents the element of $\Sel_2(E/K)$ coming from $P$, we therefore get that $V_T$ is the entire two-dimensional diagonal subspace of $E[2] \times E[2]$. Proposition \ref{MR3.3} therefore gives $d_2(E^F/K) = d_2(E/K)$ or $d_2(E^F/K) = d_2(E/K)-2$.

By Lemma \ref{kill local}, $\Sel_\phi(E^F/K) \subset \Sel_\phi(E/K)$. Since $\mathfrak{p}_1$ was chosen that so that $b$ did not localize trivially at $\mathfrak{p}_1$ and $H^1_\phi(K_{\mathfrak{p}_1}, C^F) = 0$, we get that $b \not \in Sel_\phi(E^F/K)$ and therefore that $d_\phi(E^F/K) < d_\phi(E/K) $. Lastly, since $H^1_\phi(K_v, C^F) = H^1_\phi(H_v, C)$ for all $v \ne \p_1, \p_2$ and  $H^1_\phi(K_{\mathfrak{p}_1}, C^F) =  H^1_\phi(K_{\mathfrak{p}_2}, C^F) = 0$, we get that  $\T(E^F/E^{\prime F}) < \T(E/E^\prime)$ by Theorem \ref{prodform2}.

The case where $E$ has a cyclic 4-isogeny defined over $K(E[2])$ is proved in Lemma \ref{shrinkpspec}.
\end{proof}

We can combine Proposition \ref{shrink2} and Lemma \ref{shrinkpbal} into the following Proposition.

\begin{proposition}\label{minus2}
Let $E$ be an elliptic curve defined over a number field $K$ with $E(K)[2] \simeq \mathbb{Z}/2\mathbb{Z}$ that does not possess a cyclic 4-isogeny defined over $K$ and let $t =  \max \{0, \ord_2 \T(E/E^\prime) \}$. If $d_2(E/K) \ge 2 + t$, then $E$ has a twist $E^F$ such that $d_2(E^F/K) = d_2(E/K)-2$ and $\mathcal{T}(E^F/E^{\prime F}) \le \mathcal{T}(E/E^\prime)$. In the special case where $\T(E/E^\prime) \le 1$, $E$ will have a twist $E^F$ such that $d_2(E^F/K) = d_2(E/K)-2$ and $\mathcal{T}(E^F/E^{\prime F}) \le 1$ whenever $d_2(E/K) \ge 2$.
\end{proposition}

\begin{proof} We can iteratively apply Lemma \ref{shrinkpbal} until we end up with a twist $E^{F_1}$ of $E$ with either $d_2(E^{F_1}/K) = d_2(E/K) -2$, $d_\phi(E^{F_1}/K)=1$, or $d_\phi(E^{\prime F_1}/K)=1$. If $d_2(E^{F_1}/K) = d_2(E/K)-2$, then take $E^F=E^{F_1}$ giving the result. If  $d_\phi(E^{F_1}/K) = 1$ and $d_2(E^{F_1}/K) = d_2(E/K)$, then we have $d_2(E^{F_1}/K) \ge 2$ and we can apply Proposition \ref{shrink2} to $E^{F_1}$, letting $E^F$ be the result of doing so. 

If $d_\hatphi(E^{\prime F_1}/K) = 1$ and $d_2(E^{F_1}/K) = d_2(E/K)$, then $d_\phi(E^{F_1}/K) = 1+ \ord_2 \T(E^{F_1}/E^{\prime F_1}) \le 1+ \ord_2 \T(E/E^{\prime})$. We therefore have $$d_2(E^{F_1}/K) = d_2(E/K) \ge  2 + \ord_2 \T(E/E^{\prime}) \ge 2 + \ord_2 \T(E^{F_1}/E^{\prime F_1}) = d_\phi(E^{F_1}/K) + 1$$ and are able to apply Proposition \ref{shrink2} to $E^{F_1}$, letting $E^F$ be the result of doing so.
\end{proof}

\section{Twisting to Increase Selmer Rank}\label{goingup}

We are also able to utilize the $\phi$-Selmer group to construct a twist $E^F$ of $E$ with $d_2(E^F/K) = d_2(E/K) + 2.$

\begin{lemma}\label{lemmaunbdbal}
Let $E$ be an elliptic curve defined over a number field $K$ with $E(K)[2] \simeq \mathbb{Z}/2\mathbb{Z}$ that does not possess a cyclic 4-isogeny defined over $K$. Then $E$ has a twist $E^F$ with $d_2(E^F/K) = d_2(E/K) + 2$.
\end{lemma}

\begin{proof}

This is proved here for curves that do not have a cyclic 4-isogeny defined over $K(E[2])$ and proved for curves that acquire a cyclic 4-isogeny over $K(E[2])$ in  Appendix \ref{firstapp}. The proof in the appendix works for both types of curves, but the proof presented here for curves that do not have a cyclic 4-isogeny defined over $K(E[2])$ is far less complex.

Suppose that $E$ does not have a cyclic 4-isogeny defined over $M$. By Lemma \ref{charac}, $M$ and $M^\prime$ are disjoint quadratic extensions of $K$. We may therefore find $\sigma \in G_K$ such that $\sigma \big |_M \ne 1$ and  $\sigma \big |_{M^\prime} = 1$. Let $N$ be a finite Galois extension of $K$ containing $MM^\prime$ and the ray class field $K(8\Delta_E\infty)$ such that the restriction of $\Sel_2(E/K)$ to $N$ is trivial. Let $\mathfrak{p}_1$ and $\mathfrak{p}_2$ be primes of $K$ away from $2$ where $E$ has good reduction such that $\text{Frob}_{\mathfrak{p}_1}$ in Gal$(N/K)$ is the conjugacy class of $\sigma \big |_N$ and $\text{Frob}_{\mathfrak{p}_2}$ in Gal$(N/K)$ is the conjugacy class of $\sigma^{-1} \big |_N$. Since $\sigma \sigma^{-1} \big |_{K(8\Delta_E \infty)}=1$, $\mathfrak{p}_1\mathfrak{p}_2$ has a totally positive generator $\pi \equiv 1 \text{  (mod } 8\Delta_E)$. Letting $F=K(\sqrt \pi)$, we get that all places dividing $2\Delta_E \infty$ split in $F/K$, and that $\mathfrak{p}_1$ and $\mathfrak{p}_2$ are the only primes that ramifiy in $F/K$.

Since $\gamma \big |_M \ne 1$, we get that $E(K_{\p_i}) \simeq \Zt$ and therefore that $H^1_f(K_{\p_i}, E[2]) \simeq E[2]/C$, where the isomorphism is given by evaluation at $Frob_{\p_i}$. Applying Proposition \ref{MR3.3} with $T=\{\mathfrak{p}_1, \mathfrak{p}_2\}$, we get that the localization map $\loc_T:\Sel_2(E/K) \rightarrow H^1_f(K_{\mathfrak{p}_1}, E[2]) \oplus H^1_f(K_{\mathfrak{p}_2}, E[2]) = E[2]/C \times E[2]/C$ is given by $c \mapsto \left (c(\gamma), c(\gamma^{-1})\right )$. If $c \in H^1(K, E[2])$, then we have $0 = c(1) =c(\gamma\gamma^{-1})=c(\gamma) + c(\gamma^{-1})$, giving that $c(\gamma) = c(\gamma^{-1})$. Letting $V_T$ be the image $\Sel_2(E/K)$ under the localization map, we therefore get that $V_T$ is contained in the one-dimensional diagonal subspace of $E[2] \times E[2]$. By Proposition \ref{MR3.3}, we then get  $d_2(E^F/K) = d_2(E/K)$ or $d_2(E^F/K) = d_2(E/K) + 2$.

Since all places dividing $2\Delta_E \infty$ split in $F/K$, Proposition \ref{localeq} gives that $H^1_\phi(K_v, C^F)  = H^1_\phi(K_v, C)$ for $v$ different from $\p_1$ and $\p_2$.  Since $\gamma \big |_{M^\prime} = 1$, we get that $E^\prime(K_{\p_i}) \simeq \Zt \times \Zt$. As $E(K_v)[2] \simeq \Zt$, Lemma \ref{kill local} shows that $\dimF H^1_\phi(K_{\p_i}, C^F) = 2$. Plugging this into the product formula for the Tamagawa ratio in Theorem $\ref{prodform2}$ gives $\mathcal{T}(E^F/E^{F\prime}) = 4 \mathcal{T}(E/E^\prime)$.

We now iteratively apply this process until we obtain a twist such that either $d_2(E^F/K) = d_2(E/K)+2$ or $\ord_2 \mathcal{T}(E^F/E^{F\prime}) \ge d_2(E/K)+2$. If $\ord_2 \mathcal{T}(E^F/E^{F\prime}) \ge d_2(E/K)+2$, then this would imply that $d_2(E^F/K) = d_2(E/K)+2$ as $d_2(E^F/K) \ge \ord_2  \mathcal{T}(E^F/E^{F\prime})$.

The case where $E$ acquires a cyclic 4-isogeny over $K(E[2])$ is dealt with in Lemma \ref{lemmaunbdbalspec} and Corollary \ref{propunbdbal}.
\end{proof}

If $E$ does not have a cyclic 4-isogeny defined over $M$, then by Lemma \ref{charac}, $E^\prime(K)[2] \simeq \Zt$ and $E^\prime$ does not have a cyclic 4-isogeny defined over $M^\prime$. Therefore, by swapping the roles of $E$ and $E^\prime$ and iterating the process in the above proof of Theorem \ref{lemmaunbdbal}, we are able to find a twist $E^F$ of $E$ such that $\T(E^F/E^{\prime F}) \le 1$ and $d_2(E^F/K) \simeq d_2(E/K) \pmod{2}.$ This observation allows us to prove Theorem \ref{noisogthm}.

\begin{proof}[Proof of Theorem \ref{noisogthm}]
By the above observation, we may replace $E$ with a twist such that $\T(E/E^\prime) \le 1$. Iteratively applying Proposition \ref{minus2} to $E$, we get a twist of $E$ with 2-Selmer rank $r$ for every $0 \le r  < d_2(E/K)$ with $r \equiv d_2(E/K) \pmod{2}$. Iteratively applying Lemma \ref{lemmaunbdbal} to $E$, we get a twist of $E$ with 2-Selmer rank $r$ for every $r > d_2(E/K)$ with $r \equiv d_2(E/K) \pmod{2}$.  Theorem \ref{equaltwist} then gives the result.

Further, if $E$ does not have constant 2-Selmer parity, then we obtain the result by first replacing $E$ with a twist $E^F$ such that $d_2(E^F/K) \not \equiv d_2(E/K) \pmod{2}$ and proceeding with exactly the same argument.
\end{proof}

\section{Curves which acquire a cyclic  4-isogeny over $K(E[2])$.}\label{cyclic4isogeny}
For the following section we will let $r_1$ be the number of real places of $K$ and $r_2$ be the number of complex places of $K$.

\subsection{Local conditions for curves which acquire a cyclic 4-isogeny over $K(E[2])$}
Because of Corollary \ref{characc}, the local conditions for the $\phi$-Selmer group of a curve that does not have a cyclic 4-isogeny defined over $K$ but acquires one over $K(E[2])$ satisfy much stronger conditions than those curves without a cyclic 4-isogeny defined over $K(E[2])$. Some of these restrictions have been previously been shown in \cite{BddBelow} and we simply note those here. Because of the technical nature of these results, we leave much of the work to Appendix \ref{loc4isog}.

For this section, we will assume that $E$ is a curve with $E(K)[2] \simeq \Zt$ that does not have a cyclic 4-isogeny defined over $K$ but acquires one over $K(E[2])$.

\begin{lemma}\label{mostdim1}
If $E$ has additive reduction and a place $v \nmid 2$, then $H^1_\phi(K_v, C)$ and $H^1_\hatphi(K_v, C^\prime)$ both have $\Ftwo$-dimension 1.
\end{lemma}
\begin{proof}
This is Lemma 3.3 in \cite{BddBelow}.
\end{proof}

\begin{lemma}\label{realplacec}
If $u \in K_v$ with $(u)_v < 0$ and $F = K_v(\sqrt{u})$, then exactly one of $H^1_\phi(K_v, C)$ and $H^1_\phi(K_v, C^F)$ is equal to $K_v^\times/(K_v^\times)^2$.
\end{lemma}
\begin{proof}
If $E$ is given by a model $y^2 = x^3 + Ax^2 + Bx$, then $E^F$ is given by a model $y^2 = x^3 + uAx^2 + u^2Bx$. The result then follows from Lemma \ref{realplace} 
\end{proof}

\begin{lemma}\label{splitmult}
If $E$ has split multiplicative reduction $v$ then either $H^1_\phi(K_v, C) = 0$ or  $H^1_\phi(K_v, C) =  K_v^\times/(K_v^\times)^2$. If $H^1_\phi(K_v, C) =  K_v^\times/(K_v^\times)^2$ and $F_w/K_v$ is a quadratic extension, then $H^1_\phi(K_v, C^{F_w}) = NF_w^\times/(K_v\times)^2$. If $H^1_\phi(K_v, C) = 0$, then $\dimF H^1_\phi(K_v, C^{F_w}) = 1$. 
\end{lemma}
\begin{proof}
The proof appears in Appendix \ref{loc4isog}.
\end{proof}

\begin{lemma}\label{cantwistsmall}
If $v$ is a non-complex places of $E$, then there exists a quadratic (or trivial) extension $F_w/K_v$ such that $H^1_\phi(K_v, C^{F_w}) \ne H^1(K_v, C)$.
\end{lemma}
\begin{proof}
This is a combination of Lemmas \ref{mostdim1}  \ref{realplacec}, \ref{splitmult}, and \ref{noorder4} and Corollaries \ref{nonsplit} and \ref{splitsmall}.
\end{proof}

\subsection{Proof Theorem \ref{balancedmost}}

The proof of Theorem \ref{noisogthm} required finding a twist $E^F$ of $E$ such that $\T(E^F/E^{\prime F}) \le 1$. As seen in \cite{BddBelow}, this will not alway be possible if $E$ has a cyclic 4-isogeny defined over $K$. However, we are able to prove the following:

\begin{lemma}\label{twisttor2}
If $E$ is an elliptic curve defined over $K$ that does not have a cyclic 4-isogeny defined over $K$ but acquires one over $K(E[2])$, then $E$ has a quadratic twist $E^F$ such that $\ord_2 \T(E^F/E^{\prime F}) \le r_2$. Further, if $E$ does not have constant 2-Selmer parity, then $E$ has a quadratic twist $E^{F^\prime}$ such that $\ord_2 \T(E^{F^\prime}/E^{\prime F^\prime}) \le r_2+1$ and $d_2(E^{F^\prime}/K) \not \equiv d_2(E^F/K) \pmod{2}$. 
\end{lemma}

\begin{proof}
By Lemma \ref{cantwistsmall}, we can find a quadratic (or trivial ) extension $F_w$ of each completion $K_v$ of $K$ such that $H^1_\phi(K_v, C^{F_w}) \ne H^1(K_v, C)$ for all non-complex places $v \mid 2\Delta_E\infty$. We can use an idelic construction to combine this local behavior into an single global extension.

Define an idele $\mathbf{x}$ of $K$ by $\mathbf{x}  = (x_v)$, where $x_v$ is an element of $K_v$ such that $F_w$ is given by $F_w = K_v(\sqrt{x_v})$ for each non-complex place $v$ above $2\Delta_E\infty$ and $x_v = 1$ at all other places $v$ of $K$. Let $\dd$ be the formal product of all places $v$ above $\Delta_E\infty$ and let $\gamma = [\mathbf{x}, MK(8\dd)]$ be the image of $\mathbf{x}$ under the global Artin map. If $\mathfrak{p}$ is taken to be a prime of $K$ away from $2\Delta_E$ such that $\text{Frob}_{\mathfrak{p}}$ in $MK(8\dd)$ is $\gamma$, then $\mathfrak{p}$ is principal with a generator $\pi$ such that $K_v(\sqrt{\pi}) = F_w$ for every non-complex place $v \mid 2\Delta_E\infty$. Setting $F = K(\sqrt{\pi})$, we therefore have that $\dimF H^1_\phi(K_v, C^F) \le \dimF H^1(K_v, C) - 1$ for all non-complex places $v \mid 2\Delta_E\infty$. Since $E$ has good reduction at all $v \nmid 2\p\Delta_E\infty$ and additive reduction at $\p$, Lemma \ref{mostdim1} gives that $\dimF H^1_\phi(K_v, C^F) \le \dimF H^1(K_v, C) - 1$ for all non-complex places $v$ of $K$.

By Theorem \ref{prodform2}, $\ord_2 \T(E^F/E^{\prime F})$ is given by $\displaystyle{\ord_2 \T(E^F/E^{\prime F}) = \sum_{v \text{ of } K}  \dimF H^1_\phi(K_v, C^F) - 1.}$ We then have that $$\ord_2 \T(E^F/E^{\prime F})  \le  -( r_1+r_2) + \sum_{v \mid 2}  \dimF H^1(K_v, C) - 2  $$ $$= -( r_1+r_2) + \sum_{v \mid 2} [K:\Q_2]  =  -( r_1+r_2)  + [K:\Q]  = r_2,$$ where the inequality follows from the fact that $H^1(K_v, C) \simeq K_v^\times/(K_v^\times)^2$ and the first equality further following from the fact that $\dimF H^1(K_v, C) = 2 + [K_v:\Q_2]$ for places $v \mid 2$.

By Theorem \ref{tparity}, if $E$ does not have constant 2-Selmer parity, there must be some twist $E^L$ of $E$ such that $\ord_2 \T(E/E^\prime) \not \equiv \ord_2 \T(E^L/E^{\prime L}) \pmod{2}$ and therefore some place $v_0$ such that $\dimF H^1(K_{v_0}, C) \ne \dimF H^1(K_{v_0}, C^L) \pmod{2}$. By Lemma \ref{mostdim1}, $\dimF H^1(K_v, C^L) = 1$ for all places $v \nmid 2$ where $E^L$ has additive reduction and it therefore follows that $v_0 \mid 2\Delta_E\infty$. Locally, $L/K_v$ is given by $K_v(\sqrt{u})$ where $u \in K_v^\times -  (K_v^\times)^2$. Define an idele $\mathbf{y}  = (y_v)$ of $K$ by $y_v = x_v$ for $v \ne v_0$ and $y_{v_0} = u$ if $\dimF H^1_\phi(K_{v_0}, C^F) \equiv H^1_\phi(K_{v_0}, C) \pmod{2}$ and $y_{v_0} = 1$ if $\dimF H^1_\phi(K_{v_0}, C^F) \not \equiv H^1_\phi(K_{v_0}, C) \pmod{2}$. 

Let $\dd$ be the formal product of all places $v$ above $\Delta_E\infty$ and let $\gamma^\prime = [\mathbf{y}, MK(8\dd)]$ be the image of $\mathbf{y}$ under the global Artin map. If we take a prime $\p^\prime$ of $K$ away from $2\Delta_E$ such that $\text{Frob}_{\mathfrak{p}}$ in $MK(8\dd)$ is $\gamma^\prime$, then $\p^\prime$ is principal. Let $\pi^\prime$ be a generator of $\p^\prime$ and set $F^\prime = K(\sqrt{\pi^\prime})$. Calculations similar to those above then show that $\ord_2 \T(E^{F^\prime}/E^{\prime F^\prime}) \le r_2 + 1$ and that $\ord_2 \T(E^{F^\prime}/E^{\prime F^\prime}) \not \equiv \ord_2 \T(E^F/E^{\prime F}) \pmod{2}$. Theorem \ref{tparity} then shows that $d_2(E^F/K) \not \equiv d_2(E^{F^\prime}/K) \pmod{2}$, completing the result.
\end{proof}

\begin{remark}
As the result in \cite{BddBelow} was shown by exhibiting an infinite family of curves with $\ord_2 \T(E^F/E^{\prime F}) \ge r_2$ for every curve $E$ and quadratic extension $F/K$,  Lemma \ref{twisttor2} is the best such result we can hope for without further conditions on $E$.
\end{remark}

We are now able to prove Theorem \ref{balancedmost}.

\begin{proof}[Proof of Theorem \ref{balancedmost}]
By Lemma \ref{twisttor2}, we may replace $E$ with a twist such that $\ord_2 \T(E^{F_1}/E^{\prime F_1}) \le r_2$. Iteratively applying Proposition \ref{minus2} to $E^{F_1}$, we get a twist of $E$ with 2-Selmer rank $r$ for every $\ord_2 \T(E^{F_1}/E^{\prime F_1}) \le r  < d_2(E^{F_1}/K)$ with $r \equiv d_2(E^{F_1}/K) \pmod{2}$.  Iteratively applying Lemma \ref{lemmaunbdbal} to $E^{F_1}$, we get a twist of $E$ with 2-Selmer rank $r$ for every $r \ge d_2(E^{F_1}/K)$ with $r \equiv d_2(E^{F_1}/K) \pmod{2}$. If $E$ does not have constant 2-Selmer parity, then we can additionally find a twist $E^{F_2}$ of $E$ such that $\ord_2 \T(E^{F_3}/E^{\prime F_3}) \le r_2+1$ and  $d_2(E^{F_2}/K) \not \equiv d_2(E^{F_1}/K) \pmod{2}$. Iteratively applying Proposition \ref{minus2} to $E^{F_2}$, we get a twist of $E$ with 2-Selmer rank $r$ for every $\ord_2 \T(E^{F_2}/E^{\prime F_2})  \le r  < d_2(E^{F_2}/K)$ with $r \equiv d_2(E^{F_1}/K) \pmod{2}$. Iteratively applying Lemma \ref{lemmaunbdbal} to $E^{F_2}$, we get a twist of $E$ with 2-Selmer rank $r$ for every $r \ge d_2(E^{F_2}/K)$ with $r \equiv d_2(E^{F_2}/K) \pmod{2}$.  Theorem \ref{equaltwist} applied to all of these twists then gives the result.
\end{proof}

\begin{proof}[Proof of Theorem \ref{isogpairthm}]
Parts $(i)$ and $(ii)$ of Theorem \ref{isogpairthm} are largely a consequence of the fact that $\T(E/E^\prime) = \T(E^\prime/E)^{-1}$. However, proving part $(iii)$ additionally relies on techniques similar to those used to prove Lemma \ref{twisttor2}.

Since $\T(E/E^\prime) = \T(E^\prime/E)^{-1}$, either $\T(E/E^\prime) \le 1$ or $\T(E^\prime/E) \le 1$. Therefore, by possibly exchanging the roles of $E$ and $E^\prime$, we may assume that $\T(E/E^\prime) \le 1$.  Iteratively applying Proposition \ref{minus2} to $E$ we get a twist of $E$ with 2-Selmer rank $r$ for every $0 \le r  < d_2(E/K)$ with $r \equiv d_2(E/K) \pmod{2}$. Iteratively applying Proposition \ref{propunbdbal} to $E$ we get a twist of $E$ with 2-Selmer rank $r$ for every $r > d_2(E/K)$ with $r \equiv d_2(E/K) \pmod{2}$.  Part $(i)$ then follows from applying Theorem \ref{equaltwist}.

If $E$ does not have constant 2-Selmer parity, then we can find a twist $E^F$ of $E$ with $d_2(E^F/K) \not \equiv d_2(E/K) \modtwo$. Applying part $(i)$ of this theorem to $E^F$ proves $(ii)$.

To show $(iii)$, we use an idelic argument to show that if $E$ has such a place, then $E$ has a twist $E^F$ with $d_2(E^F/K) \not \equiv d_2(E/K) \modtwo$ and $\ord_2 \T(E^F/E^{\prime F}) = \ord_2\T(E/E^\prime) \pm 1$.

Let $v_0$ be a place of $K$ such that either $v_0$ is real or $E$ has multiplicative reduction at $v_0$. Pick $u \in (K_{v_0}^\times) - (K_{v_0}^\times)^2$ such that $K_{v_0}(\sqrt{u}) \simeq \C$ if $v_0$ is a real place and $K_{v_0}(\sqrt{u})$ is unramified if $v_0$ is non-archimedean. Define an idele $\mathbf{x} = (x_v)$ of $K$ by $x_{v_0} = u$ and $x_v = 1$ for all $v \ne v_0$. Let $\dd$ be the formal product of all places $v$ above $\Delta_E\infty$ and let $\gamma = [\mathbf{x}, MK(8\dd)]$ be the image of $\mathbf{x}$ under the global Artin map. If $\mathfrak{p}$ is taken to be a prime of $K$ away from $2\Delta_E$ such that $\text{Frob}_{\mathfrak{p}}$ in $MK(8\dd)$ is $\gamma$, then $\mathfrak{p}$ is principal with a generator $\pi$. Setting $F = K(\sqrt{\pi})$, we then get that all $v \mid 2\Delta_E\infty$ different from $v_0$ split in $F/K$ and therefore that $H^1_\phi(K_v, C^F) = H^1_\phi(K_v, C)$ for all $v \mid 2\Delta_E\infty$ different from $v_0$. By Lemma \ref{mostdim1}, since $E^F$ has additive reduction at $\p$, $\dimF H^1_\phi(K_v, C^F) =1$. 

By the product formula in Theorem \ref{prodform2}, we therefore get that $$\ord_2 \T(E/E^\prime) - \ord_2 \T(E^F/E^{\prime F}) = \dimF H^1_\phi(K_{v_0}, C) - \dimF H^1_\phi(K_{v_0}, C^F).$$ If $v_0$ is a real place, then Lemma \ref{realplacec} gives that exactly one of $H^1_\phi(K_{v_0}, C)$ and  $H^1_\phi(K_{v_0}, C^F)$ has $\Ft$-dimension one and that the other has $\Ft$-dimension zero. If $v_0$ is a place where $E$ has multiplicative reduction, then exactly one of $E$ and $E^F$ has split multiplicative reduction at $v_0$ and the other has non-split multiplicative reduction at $v_0$. Applying Lemma \ref{splitmult}, we therefore get that $\dimF H^1_\phi(K_{v_0}, C) - \dimF H^1_\phi(K_{v_0}, C^F) = \pm 1$. In either case, we get that $\ord_2 \T(E^F/E^{\prime F}) = \ord_2\T(E/E^\prime) \pm 1$. Theorem \ref{tparity} then shows that $d_2(E^F/K) \not \equiv d_2(E/K) \modtwo$. 

Now, either both $\T(E/E^\prime) \le 1$ and $\T(E^F/E^{\prime F}) \le 1$ or both $\T(E^\prime/E) \le 1$ and $\T(E^{\prime F}/E^F) \le 1$. We can therefore insist that the choice of $E$ or $E^\prime$ that gives the even parities in part $(ii)$ be the same as the choice of $E$ or $E^\prime$ that gave the odd parities.
\end{proof}

\begin{proof}[Proof of Corollary \ref{rank0}]
If either $d_2(E/K) \equiv 0 \pmod{2}$ or $E$ does not have constant 2-Selmer parity, then Theorem \ref{noisogthm} and Theorem \ref{isogpairthm} imply that either $N_0(E/K, X) \gg \frac{X}{(\log X)}$ or $N_0(E^\prime/K, X) \gg \frac{X}{(\log X)}$. As $E$ and $E^\prime$ have the same Mordell-Weil rank, the result follows.
\end{proof}

\begin{proof}[Proof of Corollary \ref{rank1}]
If either $d_2(E/K) \equiv 1 \pmod{2}$ or $E$ does not have constant 2-Selmer parity,  then Theorem \ref{noisogthm} and Theorem \ref{isogpairthm} imply that either $N_1(E/K, X) \gg \frac{X}{(\log X)}$ or $N_1(E^\prime/K, X) \gg \frac{X}{(\log X)}$. Assuming conjecture $\Sha T_2(K)$ holds, then all curves with 2-Selmer rank one have Mordell-Rank one. As $E$ and $E^\prime$ have the same Mordell-Weil rank, the result follows.
\end{proof}


\bibliography{citations}  


\newpage

\addcontentsline{toc}{chapter}{APPENDICES}  \thispagestyle{empty}

\begin{center}
\noindent {\Large \textbf{APPENDICES}}
\end{center}

\appendix

\section{Proofs of Lemmas \ref{shrinkpbal} and \ref{lemmaunbdbal} when $E$ acquires a cyclic 4-isogeny over $K(E[2])$}\label{firstapp}
\subsection{Generalized Selmer Structures}
An important part of the proofs of Lemmas \ref{shrinkpbal} and \ref{lemmaunbdbal} can be thought of most naturally in terms of Selmer structures. We outline the relevant parts of theory below.

Let $K$ be a number field and $A$ a $G_K$-module. For each finite place $v$ of $K$, define the \textbf{unramified local cohomology group} $H^1_u(K_v^\text{ur}, A) \subset H^1(K_v, A)$ by $$H^1_u(K_v, A) = \ker \left ( H^1(K_v, A) \xrightarrow{res} H^1(K_v^\text{ur},  A) \right ),$$ where $K_v^\text{ur}$ is the maximal unramified extension of $K_v$.

\begin{definition} A \textbf{Selmer structure} $\mathcal{F}$ for $A$ is a collection of local cohomology subgroups $H^1_\mathcal{F}(K_v, A) \subset H^1(K_v, A)$ for each place $v$ of $K$ such that $H^1_\mathcal{F}(K_v, A)$ is equal to the unramified local subgroup $H^1_u(K_v, A)$ for all but finitely many places $v$. We define the \textbf{Selmer group} associated to the Selmer structure $\mathcal{F}$, denoted $H^1_\mathcal{F}(K, A)$ by $$H^1_\mathcal{F}(K, A) = \ker \left ( H^1(K, A) \xrightarrow{res_v} \bigoplus_{v \text { of } K}H^1(K_v, A)/H^1_\mathcal{F}(K_v, A) \right ).$$ 
\end{definition}

Let $A$ be a finitely generated $G_K$-module ramified at a finite set of primes. Define the Cartier dual $A^*$ of $A$ as $A^*=\text{Hom}(T, \mu_{p^\infty})$. For a place $v$ of $K$ we have a perfect bilinear Tate pairing $$H^1(K_v, A) \times H^1(K_v, A^*) \rightarrow H^1(K_v, \mu_{p^\infty}) \xrightarrow{\sim} \mathbb{Q}_p/\mathbb{Z}_p$$ arising from the cup-product pairing. If $\mathcal{F}$ is a Selmer structure on $A$, then we define a dual Selmer structure $\mathcal{F}^*$ on $A^*$ by setting $H^1_{\mathcal{F}^*}(K_v, A^*) = H^1_{\mathcal{F}}(K_v, A)^\bot$ for each place $v$ of $K$ where $\bot$ is taken with respect to the Tate pairing.

If  $\mathcal{F}_1$ and $\mathcal{F}_2$ are Selmer structures, we say $\mathcal{F}_1 \le \mathcal{F}_2$ if $H^1_{\mathcal{F}_1 }(K_v, A) \subset H^1_{\mathcal{F}_2 }(K_v, A)$ for each place $v$ of $K$.

\begin{theorem}[Theorem 2.3.4 in \cite{MR2}]\label{CompareSelmer}
Suppose $\mathcal{F}_1$ and $\mathcal{F}_2$ are Selmer structures with $\mathcal{F}_1 \le \mathcal{F}_2$, then 
\begin{enumerate}[(i)]
\item The sequences $$0 \rightarrow H^1_{\mathcal{F}_1 }(K, A) \rightarrow H^1_{\mathcal{F}_2 }(K, A) \xrightarrow{\sum_{res_v}}  \bigoplus_{v} H^1_{\mathcal{F}_2 }(K_v, A)/H^1_{\mathcal{F}_1 }(K_v, A)$$ and $$0 \rightarrow H^1_{\mathcal{F}_2^*}(K, A^*) \rightarrow H^1_{\mathcal{F}_1^* }(K, A^*) \xrightarrow{\sum_{res_v}}  \bigoplus_{v} H^1_{\mathcal{F}_1^* }(K_v, A^*)/H^1_{\mathcal{F}_2^* }(K_v, A^*)$$ are exact where the sum is taken over all places $v$ such that $H^1_{\mathcal{F}_1 }(K_v, A) \ne H^1_{\mathcal{F}_2 }(K_v, A)$.
\item The images of the right hand maps are orthogonal complements with respect to the sum of the local Tate pairings.
\end{enumerate}
\end{theorem}
\begin{proof}
The first part follows immediately from the definition of Selmer structures. The second is part of Poitou-Tate global duality; see Theorem I.4.10 in \cite{Milne} for example.
\end{proof}

Both the 2-Selmer and $\phi$-Selmer groups arise from Selmer structures. In the case of the 2-Selmer group, since $E[2]^* = E[2]$ under the Weil pairing and $H^1_f(K_v, E[2])$ is self-dual under the local Tate-pairing, we get that the Selmer group associated to the dual Selmer structure for $\Sel_2(E/K)$ is equal to $\Sel_2(E/K)$. Lemma 3.2 in \cite{MR} referenced in Section \ref{methodsofrubingandmazur} then an application of Theorem \ref{CompareSelmer}. 

\subsection{Proof of Lemma \ref{shrinkpbal}}
\begin{lemma}\label{shrinkpspec}
Let $E$ be an elliptic curve defined over a number field $K$ with $E(K)[2] \simeq \mathbb{Z}/2\mathbb{Z}$ that does not possess a cyclic 4-isogeny defined over $K$ but acquires one over $K(E[2])$. If both $d_\phi(E/K) > 1$ and $d_{\hat \phi}(E^\prime/K) > 1$, then $E$ has a quadratic twist $E^F$ such that 
\begin{enumerate}[(a)]
\item $d_\phi(E^F/K) < d_\phi(E/K) $, 
\item $d_2(E^F/K) = d_2(E/K)$ or $d_2(E^F/K) = d_2(E/K)-2$, and
\item $\T(E^F/E^{\prime F}) \le \T(E/E^\prime)$
\end{enumerate}
\end{lemma}
\begin{proof}
Let $M = K(E[2])$. We begin by recalling that $K(E^\prime[2]) = M$ by Lemma \ref{charac} and that therefore $\dimF E(K_v)[2] = \dimF E^\prime(K_v)[2]$ for all places $v$ of $K$ by Corollary \ref{characc}.

Pick $R \in E(\overline{K})$ with $2R = P$ and define a co-cycle $c_0 : G_K \rightarrow E[2]$ by $c_0(\sigma) = \sigma(R) - R$ for $\sigma \in G_K$. As $E$ does not have a cyclic 4-isogeny defined over $K$, there exists some $\gamma_1 \in G_K$ such that $c_0(\gamma_1) \not \in C$.  Observing that $$\phi(c_0(\sigma)) = \phi(\sigma(R) - R) = \phi(\sigma(R)) -\phi(R) = \sigma(\phi(R)) - \phi(R)$$ for $\sigma \in G_K$ and recalling that $\phi(R) \in E^\prime[2]$, we see that $c_0(\sigma) \in C$ for all $c \in G_{M}$ and we therefore get that $\gamma_1 \not \in G_M$.

Since  $d_\phi(E/K) > 1$ and $d_{\hat \phi}(E^\prime) > 1$, we can find $b \in \Sel_\phi(E/K)$ and $\hat b \in \Sel_{\hat\phi}(E^\prime/K)$ such that $b$ does not come from $C^\prime$ and that $\hat b$ does not come from $C$.

Take $N_{b} = \overline{K}^{\text{Ker}b}$ and  $N_{\hat b} = \overline{K}^{\text{Ker} \hat b}$. Since $b$ does not come from $C^\prime$ and $\hat b$ does not come from $C$, $N_{b}$ and  $N_{\hat b}$ are (not necessarily distinct) quadratic extensions of $K$ different from $M$. We can therefore find $\gamma_2 \in G_K$ such that $\gamma_2 \big |_{M} = 1$, $\gamma_2 \big |_{N_b} \ne 1$, and $\gamma_2 \big |_{N_{\hat b}} \ne 1$. Set $\gamma_3 = (\gamma_1\gamma_2)^{-1}\in G_K$ and observe that $\gamma_3 \big |_M \ne 1$.

Let $N$ be a finite Galois extension of $K$ containing $MK(8\Delta_E\infty)$ such that the restrictions of $\Sel_2(E/K)$, $\Sel_\phi(E/K)$, and $\Sel_\hatphi(E^\prime/K)$ to $N$ are zero. Choose 3 primes $\mathfrak{p}_1, \mathfrak{p}_2, \text{ and } \mathfrak{p}_3$ not dividing 2 such that $E$ has good reduction at each $\mathfrak{p}_i$ and $\text{Frob}_{\mathfrak{p_i}} \big |_N = \gamma_i \big |_N$ for $i = 1, ..., 3$. Since $\gamma_1\gamma_2\gamma_3 \big |_{K(8\Delta_E\infty)}=1$, we get that $\mathfrak{p}_1\mathfrak{p}_2\mathfrak{p}_3$ has a totally positive generator $\pi$ with $\pi \equiv 1\text{  }\pmod{8\Delta_E\infty}$. Letting $F=K(\sqrt{\pi})$, we get that all places dividing $2\Delta_E\infty$ split in $F/K$, and that $\mathfrak{p}_1, \mathfrak{p}_2, \text{ and } \mathfrak{p}_3$ are the only primes that ramify in $F/K$.

We now apply Proposition \ref{MR3.3}  with $T = \{\mathfrak{p}_1, \mathfrak{p}_2, \mathfrak{p}_3\}$. Since $E$ has good reduction at $\mathfrak{p}_1$, $\mathfrak{p}_2$, and $\mathfrak{p}_3$ it then follows that $H^1_f(K_{\mathfrak{p}_1}, E[2]) = E[2]/(\gamma_1-1)E[2] = E[2]/C$,  $H^1_f(K_{\mathfrak{p}_2}, E[2]) = E[2]/(\gamma_2-1)E[2] = E[2]$,  and $H^1_f(K_{\mathfrak{p}_3}, E[2]) = E[2]/(\gamma_3-1)E[2] = E[2]/C$ and that the localization map $\loc_T:\Sel_2(E/K) \rightarrow H^1_f(K_{\mathfrak{p}_1}, E[2]) \oplus H^1_f(K_{\mathfrak{p}_2}, E[2]) \oplus H^1_f(K_{\mathfrak{p}_3}, E[2])$ is given by $c \mapsto \left (c(\gamma_1), c(\gamma_2), c(\gamma_3)\right )$.

As $b \in \Sel_\phi(E/K)$, we may view it as an element of $\Sel_2(E/K)$. Viewing $b$ as a co-cycle, we get that $b(\gamma_2) \ne 0$ since $\gamma_2 \big |_{N_{b}} \ne 1$ and that $b(\gamma_2) \in C$ since $b \in H^1(K, C)$. Let $V_T$ be the image $\Sel_2(E/K)$ under the localization map. By design, $c_0(\gamma_1) \not \in C$. As $c_0$ represents the element of $\Sel_2(E/K)$ coming from $P$, we therefore get that $V_T$ has dimension at least 2 in the 4 dimensional space $H^1_f(K_{\mathfrak{p}_1}, E[2]) \oplus H^1_f(K_{\mathfrak{p}_2}, E[2]) \oplus H^1_f(K_{\mathfrak{p}_3}, E[2])$. Noting further that if $c$ is a co-cycle representing an element of $H^1(K, E[2])$, then we have $$0 = c(\gamma_1\gamma_2\gamma_3) = c(\gamma_1) + c(\gamma_2) + c(\gamma_3) + C,$$ we get that $(Q, 0, 0) \not \in V_T$, so that $V_T$ has $\Ftwo$ dimension at most 3. Proposition \ref{MR3.3} therefore gives  $d_2(E^F/K) = d_2(E/K)$ or $d_2(E^F/K) = d_2(E/K)-2$.

By Lemma \ref{localeq}, $H^1_\phi(K_v, C^F) = H^1_\phi(K_v, C)$ and $H^1_{\hat \phi}(K_v, C^{\prime F}) = H^1_{\hat \phi}(K_v, C^\prime)$ for all $v$ different from $\mathfrak{p}_2$. Define strict and relaxed Selmer groups $Z_T$ and $Z^T$ as $$Z_T =\ker \left ( \Sel_\phi(E/K) \xrightarrow{res_{\p_2}} H^1(K_{\p_2}, C) \right )$$ and $$Z^T =\ker \big ( H^1(K, C) \xrightarrow{\sum_{res_v}} \bigoplus_{v \ne \p_2} H^1(K_v, C)/H^1_\phi(K_v, C) \big ).$$ Observe that $Z_T \subset \Sel_\phi(E/K), \Sel_\phi(E^F/K) \subset Z^T$ and $Z^{T*} \subset \Sel_\phihat(E^\prime/K) , \Sel_\phihat(E^{\prime F}/K) \subset Z_T^*$, where $Z_T^*$ and $Z^{T*}$ are the Selmer groups associated to the dual Selmer structures for $Z_T$ and $Z^T$ respectively.

We now apply Theorem \ref{CompareSelmer} to the Selmer structures defining $Z_T$ and $Z^T$, getting that $$0 \rightarrow Z_T \rightarrow Z^T \xrightarrow{res_{\p_2}} H^1(K_{\p_2}, C)$$ and $$0 \rightarrow Z^{T*} \rightarrow Z_T^* \xrightarrow{res_{\p_2}} H^1(K_{\p_2}, C^\prime)$$ are exact and that the images of the right hand sides are exact orthogonal complements with respect to the local Tate pairing. This means that $\dimF Z^T/Z_T + \dimF Z_T^*/Z^{T*} = 2$. As $b$ is non-zero in $Z^T/Z_T$ and $\hat b$ is non-zero in $Z_T^*/Z^{T*}$, we get that $\dimF Z^T/Z_T = \dimF Z_T^*/Z^{T*} = 1$. By Lemma \ref{twistbp}, $H^1_\phi(K_{\p_2}, C^F) \cap H^1_\phi(K_{\p_2}, C) = 0$, so $b$ does not restrict into $H^1_\phi(K_{\p_2}, C^F)$. Since $b$ generates $Z^T/Z_T$, this tells us that $Z^T \cap \Sel_\phi(E^F/K) = Z_T$. Further, $\Sel_\phi(E/K) = Z^T$ since $b$ generates $Z^T/Z_T$. We therefore get that $d_\phi(E^F/K) = d_\phi(E/K) -1$.

Finally, since $H^1_\phi(K_{\mathfrak{p}_2}, C^F)$  has $\Ftwo$ dimension 1 by Lemma \ref{twistbp}, the product formula in Lemma \ref{prodform2} gives us that $\T(E^F/E^{\prime F}) = \T(E/E^\prime)$.
\end{proof}

\subsection{Proof of Lemma \ref{lemmaunbdbal}}
\begin{lemma}\label{lemmaunbdbalspec}
Let $E$ be an elliptic curve defined over a number field $K$ with $E(K)[2] \simeq \mathbb{Z}/2\mathbb{Z}$ that does not possess a cyclic 4-isogeny defined over $K$. Then $E$ has a twist $E^F$ with $d_\phi(E^F/K) =  d_\phi(E/K) +1$ such that $d_2(E^F/K) = d_2(E/K)$ or $d_2(E^F/K) = d_2(E/K) + 2$ and $\T(E^F/E^{\prime F}) = \T(E/E^\prime)$.
\end{lemma}

\begin{proof}
This lemma is proved by finding two quadratic extension $F_1/K$ and $F_2/K$ such that the Selmer structures for $\Sel_\phi(E^{F_1}/K)$ and $\Sel_\phi(E^{F_2}/K)$ differ from each other and from the Selmer structure for $\Sel_\phi(E/K)$ at a single place $\p$. We are then able to show that exactly one of the twists $E^{F_1}$ and $E^{F_2}$ of $E$ has the same $\phi$-Selmer rank as $E$ and that the other twist has $\phi$-Selmer rank one greater that the $\phi$-Selmer rank as $E$.

Letting $M = K(E[2])$ and $M^\prime = K(E^\prime[2])$. As $E$ does not have a cyclic 4-isogeny defined over $K$, part $(i)$ of Lemma \ref{charac} gives us $M^\prime \ne K$, and we can therefore find $\gamma \in \Gal(MM^\prime/K)$ such that $\gamma \big |_M \ne 1$ and  $\gamma \big |_{M^\prime} \ne 1$. Let $N$ be any finite extension of $MM^\prime$ that is Galois over $K$ such that the restrictions of $\Sel_2(E/K)$, $\Sel_\phi(E/K)$, and $\Sel_\hatphi(E^\prime/K)$ to $N$ are trivial. Choose $\sigma_1 \in \text{Gal}(NK(8\Delta_E\infty)/K)$ such that $\sigma_1 \big |_{MM^\prime} = \gamma$. Now choose two primes, $\mathfrak{q}_1$ and $\mathfrak{q}_2$ away from $2\Delta_E$ such that $Frob_{\mathfrak{q}_1} \big |_{NK(8\Delta_E\infty)} = \sigma_1$ and $Frob_{\mathfrak{q}_2} \big |_{NK(8\Delta_E\infty)} = (\sigma_1)^{-1}$. As $Frob_{\mathfrak{q}_1} Frob_{\mathfrak{q}_2} \big |_{K(8\Delta_E\infty)} = 1$, we get that $\mathfrak{q}_1\mathfrak{q}_2$ has a totally positive generator $\pi^\prime$ with $\pi^\prime \equiv 1 (\text{mod } 8\Delta_E)$. Define $L = K(\sqrt{\pi^\prime})$.

Now take $\sigma_2 \in \text{Gal}(NLK(8\Delta_E\infty)/K)$ such that $\sigma_2 \big |_{NK(8\Delta_E\infty)}=1$ and $\sigma_2 \big |_L \ne 1$. Since $\mathfrak{p}_1$ and $\mathfrak{p}_2$ are ramified in $L/K$ and not in $NK(8\Delta_E\infty)/K$, we have $L \cap NK(8\Delta_E\infty)/K = K$ and we can therefore always find such a $\sigma_2$. Let $\mathfrak{p}$ be any prime away from $2\Delta_E\mathfrak{q}_1\mathfrak{q}_2$ such that the image of $Frob_{\mathfrak{p}}$ in $\text{Gal}(NLK(8\Delta_E\infty)/K)$ is $\sigma_2$. As $Frob_{\mathfrak{p}}$ in $\text{Gal}(K(8\Delta_E\infty)/K)$ is trivial, $\mathfrak{p}$ has a totally positive generator $\pi$ with $\pi \equiv 1\pmod{8\Delta_E}$. 

Define quadratic extensions $F_1/K$ and $F_2/K$ by $F_1 = K(\sqrt{\pi})$ and $F_2 = K(\sqrt{\pi\pi^\prime})$. Observe that all places above $2\Delta_E\infty$ split in both $F_1/K$ and $F_2/K$. The only prime ramified in $F_1/K$ is $\mathfrak{p}$ and the only primes ramified in $F_2/K$ are $\mathfrak{q}_1$, $\mathfrak{q}_2$, and $\mathfrak{p}$. 

We then apply Proposition \ref{MR3.3} to $E^{F_1}$ and $E^{F_2}$ with $T_1=\{ \mathfrak{p} \}$ and $T_2 = \{\mathfrak{q}_1, \mathfrak{q}_2, \mathfrak{p}\}$ respectively. Since the images of $Frob_{\mathfrak{q}_1}$ and $Frob_{\mathfrak{q}_2}$ in $\text{Gal}(M/K)$ are non-trivial, $E(K_{\mathfrak{p}_1})[2] \simeq \mathbb{Z}/2\mathbb{Z} \simeq E(K_{\mathfrak{p}_2})[2]$. Since $Frob_{\p} \big |_M = 1$, $E(K_\p)[2] \simeq \Zt \times \Zt$. The localization map $\loc_{T_1}:\Sel_2(E/K) \rightarrow H^1_f(K_\p, E[2])$ is therefore given by $c \mapsto \left (c(\sigma_2)\right )$ and the localization map $\loc_{T_2}:\Sel_2(E/K) \rightarrow H^1_f(K_{\q_1}, E[2]) \oplus H^1_f(K_{\q_2} ,E[2])\oplus H^1_f(K_{\p} E[2])$ is therefore given by $c \mapsto \left (c(\sigma_1), c(\sigma_1^{-1}), c(\sigma_2) \right )$. As $0 = c(1) = c(\sigma_1\sigma_1^{-1}) = c(\sigma_1) + \sigma_1c(\sigma_1^{-1})$, it follows that $c(\sigma_1) = c(\sigma_1^{-1})$ in $E[2]/C$.

Since the image of $Frob_{\mathfrak{p}}$ in $\text{Gal}(N/K)$ is trivial, we get that $G_{K_{\mathfrak{p}}} \subset G_N$. Therefore $\loc_{T_1} \left ( \Sel_2(E/K) \right )$ is trivial and $\loc_{T_2} \big |_{\Sel_2(E/K)}$ is given by $c \mapsto \left (c(\sigma_1), c(\sigma_1), 0 \right )$. 

Pick $R \in E(\overline K)$ such that $2R = P$ and define a co-cycle $c_0 : G_K \rightarrow E[2]$ by $c_0(\sigma) = \sigma(R) - R$ for $\sigma \in G_K$. Observing that $$\phi(c_0(\sigma)) = \phi(\sigma(R) - R) = \phi(\sigma(R)) -\phi(R) = \sigma(\phi(R)) - \phi(R)$$ for $\sigma \in G_K$, we see that $c_0(\sigma) \in C$ if and only if $\phi(R)$ is fixed by $\sigma$. As $\phi(R) \in E^\prime[2] - C^\prime$ and $\sigma_1 \big |_{M^\prime} \ne 1$, we therefore see that $c_0(\sigma) \not \in C$. As $c_0$ represents the element of $\Sel_2(E/K)$ coming from $C$, we therefore get that $\loc_{T_2} \left (\Sel_2(E/K) \right ) = \langle (Q, Q , 0 ) \rangle$. By Proposition \ref{MR3.3}, we then get that $d_2(E^{F_i}/K) = d_2(E/K)$ or $d_2(E^{F_i}/K) = d_2(E/K) + 2$ for $i = 1, 2$.

By Lemma \ref{localeq}, the local conditions away from $\mathfrak{p}$ for both $\Sel_\phi(E^{F_1}/K)$ and $\Sel_\phi(E^{F_2}/K)$ are identical to those for $\Sel_\phi(E/K)$. Set $T = \{ \mathfrak{p} \}$ and define strict and relaxed Selmer groups $Z_T$ and $Z^T$ as $$Z_T = \ker \left ( \Sel_\phi(E/K) \xrightarrow{res_\p} H^1(K_{\p}, C) \right )$$ and $$Z^T = \ker \big ( \Sel_\phi(E/K) \xrightarrow{res_v} \bigoplus_{v \ne \p} H^1(K_v, C)/H^1_\phi(K_v, C) \big ).$$ Observe that $Z_T \subset \Sel_\phi(E/K), \Sel_\phi(E^F/K) \subset Z^T$ and $Z^{T*} \subset \Sel_\phihat(E^\prime/K) , \Sel_\phihat(E^{\prime F}/K) \subset Z_T^*$,  where $Z_T^*$ and $Z^{T*}$ are the Selmer groups associated to the dual Selmer structures for $Z_T$ and $Z^T$ respectively. Moreover, as the image of $Frob_{\mathfrak{p}}$ in $\text{Gal}(N/K)$ is trivial, we get that $\text{res}_{\mathfrak{p}}\left ( \Sel_\phi(E/K) \right )$ is trivial, showing that $\Sel_\phi(E/K) = Z_T$. Similarly, $Sel_\phihat(E^\prime/K) = Z^{T*}$.

Since $\sigma_2 \big |_{MM^\prime} = 1$ we have both $E(K_{\mathfrak{p}})[2] \simeq \Zt \times \Zt$ and $E^\prime(K_{\mathfrak{p}})[2] \simeq \Zt \times \Zt$. As $Frob_{\mathfrak{p}}$ is non-trivial in $L=K(\sqrt{\pi^\prime})$, we get that $\mathfrak{p}$ doesn't split in $L$ and therefore $x^2 - \pi^\prime$ is irreducible mod $\mathfrak{p}$ which means that $\pi^\prime \not \in (K_{\mathfrak{p}}^\times)^2$. Locally, we therefore get that $F_1$ and $F_2$ give different ramified extensions of $K_\p$. By Lemma \ref{twistbp}, we therefore get that $H^1(K_{\mathfrak{p}} , C)$ are $H^1_\phi(K_{\mathfrak{p}}, C^{F_1})$, $H^1_\phi(K_{\mathfrak{p}}, C^{F_2})$, and $H^1_u(K_{\mathfrak{p}}, C)$ are the three distinct one-dimensional subspaces of $H^1(K_{\mathfrak{p}} , C)$.

We now apply Theorem \ref{CompareSelmer} to the Selmer structures defining $Z_T$ and $Z^T$, getting that $$0 \rightarrow Z_T \rightarrow Z^T \xrightarrow{res_{\p_2}} H^1(K_{\p_2}, C)$$ and $$0 \rightarrow Z^{T*} \rightarrow Z_T^* \xrightarrow{res_{\p_2}} H^1(K_{\p}, C)$$ are exact and that the images of the right hand sides are exact orthogonal complements in $H^1(K_{\p}, C)$. This means that $\dimF Z^T/Z_T + \dimF Z_T^*/Z^{T*} = 2$. As $\Sel_\phi(E/K)$ and $\Sel_\hatphi(E^\prime/K)$ restrict trivially into  $H^1(K_{\mathfrak{p}} , C)$ and $H^1(K_{\mathfrak{p}} , C^\prime)$ respectively and both $\dimF Z^T/\Sel_\phi(E/K)$ and $\dimF Z_T^*/\Sel_\hatphi(E^\prime/K)$ are less than or equal to 1, it follows that neither $Z^T$ surjects  onto $H^1(K_{\mathfrak{p}} , C)$ nor $Z_T^*$ surjects onto $H^1(K_{\mathfrak{p}} , C^\prime)$. This means that both $Z^T/Z_T$ and $Z_T^*/Z^{T*}$ have $\mathbb{F}_2$ dimension 1 and therefore $res_{\mathfrak{p}} (Z^T)$ is a dimension 1 subspace of $H^1(K_{\mathfrak{p}} , C)$. 

Since $Z_T = \Sel_\phi(E/K)$, $res_{\mathfrak{p}} (Z^T)$ must be either $H^1_\phi(K_{\mathfrak{p}}, C^{F_1})$ or $H^1_\phi(K_{\mathfrak{p}}, C^{F_2})$. Choosing $F$ from among $F_1$ and $F_2$ such that $res_{\mathfrak{p}} (Z^T) =  H^1_\phi(K_{\mathfrak{p}}, C^F)$, we get that $Z^T = \Sel_\phi(E^F/K)$. We therefore get that $d_\phi(E^F/K) = d_\phi(E/K) + 1$.

Finally, since $H^1_\phi(K_\mathfrak{p}, C^F)$  has $\Ftwo$ dimension 1 by part $(ii)$ of Lemma \ref{twistbp}, the product formula in Lemma \ref{prodform2} gives us that $\T(E^F/E^{\prime F}) = \T(E/E^\prime)$.
\end{proof}

\begin{corollary}\label{propunbdbal}
Let $E$ be an elliptic curve defined over a number field $K$ with $E(K)[2] \simeq \mathbb{Z}/2\mathbb{Z}$ that does not possess a cyclic 4-isogeny defined over $K$. Then $E$ has a twist $E^F$ with $d_2(E^F/K) = d_2(E/K) + 2$.
\end{corollary}
\begin{proof}
Iteratively apply Lemma \ref{lemmaunbdbalspec} to $E$ until either $d_2(E^F/K) = d_2(E/K)+2$ or $d_\phi(E^F/K) = d_2(E/K)+3$. Since $E$ does not have a cyclic 4-isogeny defined over $K$, the map $\Sel_\phi(E/K) \rightarrow \Sel_2(E/K)$ is 2-to-1 by Theorem \ref{gss} and therefore $d_2(E^F/K) \ge d_\phi(E^F/K) -2$. If $d_\phi(E^F/K) = d_2(E/K)+3$, then $d_2(E^F/K) > d_2(E/K)$, and since $d_2(E^F/K) = d_2(E/K)$ or $d_2(E^F/K) = d_2(E/K)+2$, it must be in fact that $d_2(E^F/K) = d_2(E/K)+2$.
\end{proof}

\section{Local conditions for curves which acquire a cyclic 4-isogeny over $K(E[2])$}\label{loc4isog}

For this section, we will assume that $E$ is a curve with $E(K)[2] \simeq \Zt$ that does not have a cyclic 4-isogeny defined over $K$ but acquires one over $K(E[2])$.

\begin{lemma}\label{realplace}
If $v$ is a real place of $K$ and $E$ is given by a model $y^2 = x^3 + Ax^2 + Bx$, then $H^1_\phi(K_v, C)$ has $\Ftwo$-dimension 1 if $(A)_v < 0$ and $H^1_\phi(K_v, C) = 0$ if $(A)_v > 0$.
\end{lemma}
\begin{proof}
The discriminant for this model of $E$ is given by $\Delta_E = 16(A^2 - 4B)B^2$ and the discriminant of the corresponding model for $E^\prime$ is given by $\Delta_ {E^\prime} = 256(A^2 - 4B)^2B$. By Corollary \ref{characc}, either both $(B)_v > 0$ and $(A^2 - 4B)_v > 0$ or both  $(B)_v < 0$ and $(A^2 - 4B)_v < 0$. As the image of $C^\prime$ in $H^1_\phi(K_v, C)$ is generated by $\Delta_E (K_v^\times)^2$ and the image of $C$ in $H^1_\hatphi(K_v, C^\prime)$ is generated by $\Delta_{E^\prime}(K_v^\times)^2$, the latter case can't occur as it would violate Lemma \ref{localduality}. We therefore get that $(B)_v > 0$ and $(A^2 - 4B)_v > 0$.

Suppose $(A)_v < 0$. If $x_0 \in K_v$ with $(x_0)_v < 0$, then each of  $(x_0^3)_v, (Ax_0^2)_v, \text{ and } (Bx_0)_v$ are negative and therefore, $x_0$ can't appear as the $x$-coordinate of any point on $E(K_v)$. Since $C$ has trivial image in $K_v^\times/(K_v^\times)^2$, we therefore get that $H^1_{\hat \phi}(K_v, C^\prime) = 0$. Lemma \ref{localduality} then gives that $H^1_{\phi}(K_v, C) = K_v^\times/(K_v^\times)^2$.
Now suppose $(A)_v > 0$. The curve $E^\prime$ is given by a model $y^2 = x^3 - 2Ax^2 +(A^2 - 4B)x$.  Exchanging the roles of $E$ and $E^\prime$ in the above shows that $H^1_{\phi}(K_v, C)=0$ and $H^1_{\hat \phi}(K_v, C^\prime) = K_v^\times/(K_v^\times)^2$.
\end{proof}

\begin{LemmaSplitMult}
If $E$ has split multiplicative reduction $v$ then either $H^1_\phi(K_v, C) = 0$ or  $H^1_\phi(K_v, C) =  K_v^\times/(K_v^\times)^2$. If $H^1_\phi(K_v, C) =  K_v^\times/(K_v^\times)^2$ and $F_w/K_v$ is a quadratic extension, then $H^1_\phi(K_v, C^{F_w}) = NF_w^\times/(K_v\times)^2$. If $H^1_\phi(K_v, C) = 0$, then $\dimF H^1_\phi(K_v, C^{F_w}) = 1$. 
\end{LemmaSplitMult}
\begin{proof}
Lemma 3.4 in \cite{BddBelow} states that if $E$ has split multiplicative reduction at $v$ with Kodaira symbol $I_{n}$ and $E^\prime$ has Kodaira symbol $I_{2n}$ at $v$, then $H^1_\phi(K_v, C) =  K_v^\times/(K_v^\times)^2$ and $H^1(K_v, C^F)$ is given by $H^1(K_v, C^F) = N_{F_w/K_v} F_w^\times / (K_v^\times)^2$ for quadratic extensions $F_w/K_v$.

Since $E$ and $E^\prime$ have split multiplicative reduction at $v$, $E/K_v$ and $E^\prime/K_v$ are $G_{K_v}$ isomorphic to Tate curves $E_q$ and $E_{q^\prime}$ respectively. The curve $E_q$ can be two-isogenous to three different curves: $E_{q^2}, E_{\sqrt{q}}$, and $E_{-\sqrt{q}}$. We therefore get that $|q^\prime|_v = |q|_v^2$ or $|q^\prime|_v = |q|_v^\frac{1}{2}$. Since the Kodaira symbol of $E$ is $I_n$, where $n = \ord_v q$, it then follows that if $E$ has Kodaira symbol $I_{n}$, then $E^\prime$ has Kodaira symbol $I_{2n}$ or $E^\prime$ has Kodaira symbol $I_{\frac{n}{2}}$.

In the case where $E$ has Kodaira symbol $I_{n}$ and $E^\prime$ has Kodaira symbol $I_{2n}$, we get that $H^1_\phi(K_v, C) =  K_v^\times/(K_v^\times)^2$ and $H^1_\phi(K_v, C^{F_w}) = NF_w^\times/(K_v\times)^2$ by Lemma 3.4 in \cite{BddBelow}. In the case where $E$ has Kodaira symbol $I_{n}$ and $E^\prime$ has Kodaira symbol $I_{\frac{n}{2}}$ we get that $H^1_\hatphi(K_v, C^\prime) =  K_v^\times/(K_v^\times)^2$ and $H^1_\hatphi(K_v, C^{\prime F_w}) = NF_w^\times/(K_v\times)^2$ by applying Lemma 3.4 in \cite{BddBelow} with the roles of $E$ and $E^\prime$ exchanged. Applying Lemma \ref{localduality} completes the result.
\end{proof}

\begin{corollary}\label{nonsplit}
If $E$ has non-split multiplicative reduction at a place $v$, then $H^1_\phi(K_v, C) \ne H^1(K_v, C)$.
\end{corollary}
\begin{proof}
Since $E$ has non-split multiplicative reduction at $v$, there is some quadratic extension $F_w/K_v$ such that $E^{F_w}$ has split multiplicative reduction at $v$. As $E$ is the twist of $E^{F_w}$ by $F_w/K_v$, applying Lemma \ref{splitmult} to $E^{F_w}$ gives that either $H^1_\phi(K_v, C) = NF_w^\times/(K_v\times)^2$ or that $\dimF H^1_\phi(K_v, C) = 1$, dependent on whether $H^1_\phi(K_v, C^{F_w}) =  K_v^\times/(K_v^\times)^2$ or $H^1_\phi(K_v, C^{F_w}) =  0$. 
\end{proof}

\begin{corollary}\label{splitsmall}
If $E$ has split multiplicative reduction at a place $v$ and $F = K_v(\sqrt{u})$ where $u \in O_{K_v}^\times - (O_{K_v}^\times)^2$, then $H^1_\phi(K_v, C^F) \ne  K_v^\times/(K_v^\times)^2$.
\end{corollary}
\begin{proof}
The curve $E^F$ has non-split multiplicative reduction at $v$ and the result then follows from Corollary \ref{nonsplit}
\end{proof}

\begin{lemma}\label{noorder4}
If $v$ is a place above 2, then there exists a quadratic extension $F/K_v$ such that $H^1_\phi(K_v, C^{F}) \ne K_v^\times/(K_v^\times)^2$.

\end{lemma}
\begin{proof}
We begin by showing that if $v$ is a place above 2, then there exists a quadratic extension $F/K_v$ such that $E^{\prime F}(K_v)$ has no points of order 4.

Pick a model $y^2 = x^3 + Ax^2 + Bx$ for $E$. The $y$-coordinates of the points of order 4 on $E^\prime(\overline{K_v})$ are given by $y_1, y_2, ..., y_6, -y_1,...,-y_6$ for some $y_1, ..., y_6 \in \overline{K_v}$. Since $v \mid 2$, there are at least 7 non-trivial quadratic extensions of $K_v$ and it therefore follows that there is some quadratic $F/K_v$ such that $F$ is not generated over $K_v$ by any of the $y_i$. We then get that $E^\prime(F)^{\sigma = -1}$ contains no points of order 4, where $\sigma$ generates $\Gal(F/K_v)$. As $E^{\prime F}(K_v) \simeq E^\prime(F)^{\sigma = -1}$, we get that $E^{\prime F}(K_v)$ has no points of order 4.

If $E^F(K_v)[2] \simeq \Zt$, then $\Delta_{E^\prime}$ has a non-trivial image in $H^1(K, C)$, since $\dimF E^F(K_v)[2] \simeq \dimF E^{\prime F}(K_v)[2]$ and therefore $H^1_\hatphi(K_v, C^{\prime F}) \ne 0$. Suppose $E^F(K_v)[2] \simeq \Zt \times \Zt$. Let $Q \in E^F(K_v)[2] - C^F$.  The image of $Q$ in $H^1(K, C)$ is trivial if and only if $Q = \phihat(R)$ for some $R \in E^{\prime F}(K_v)$. However, since $Q \in E^F(K_v)[2] - C^F$, $R$ would need to have order 4. As $E^{\prime F}(K_v)$ contains no points of order 4, the image of $Q$ is non-trivial in $H^1(K, C)$ and therefore $H^1_\hatphi(K_v, C^{\prime F}) \ne 0$. Applying Lemma \ref{localduality}, we get that $H^1_\phi(K_v, C^{F}) \ne K_v^\times/(K_v^\times)^2$.
\end{proof}
\end{document}